\documentclass[12pt]{amsart}

\usepackage{amsmath, amssymb, amscd, verbatim, xspace,amsthm}
\usepackage{latexsym, epsfig, color}

\newcommand{\A}{\ensuremath{{\mathbb{A}}}}

\newcommand{\Z}{\ensuremath{{\mathbb{Z}}}\xspace}
\renewcommand{\P}{\ensuremath{{\mathbb{P}}}}

\newcommand{\F}{\ensuremath{{\mathbb{F}}}}

\newcommand{\ra}{\rightarrow}

\newcommand\im{\operatorname{im}}

\newcommand\Prob{\operatorname{Prob}}

\newcommand\tensor{\otimes}
\newcommand\isom{\cong}
\newcommand\sub{\subset}
\newcommand\tesnor{\otimes}

\newcommand\Spec{\operatorname{Spec}}

\newcommand\cO{\mathcal{O}}

\newcommand\bq{\begin{equation}}
\newcommand\eq{\end{equation}}

\newtheorem{proposition}{Proposition}[section]
\newtheorem{theorem}[proposition]{Theorem}
\newtheorem{corollary}[proposition]{Corollary}

\newtheorem{lemma}[proposition]{Lemma}

\theoremstyle{remark}
\newtheorem{remark}[proposition]{Remark}
\newtheorem{example}[proposition]{Example}

\usepackage[headings]{fullpage}
\usepackage[all]{xy}

\newcommand\reg{\operatorname{reg}}

\newtheorem{nts}{Note to self}

\newcommand{\defi}[1]{\textsf{#1}} 

\newcommand{\Fbar}{{\overline{\F}}}

\newcommand{\calQhigh}{{\mathcal Q}^{\operatorname{high}}}

\newcommand{\II}{{\mathcal I}}

\newcommand{\red}{{\operatorname{red}}}

\newcommand{\olP}{{\overline{P}}}

\newcommand{\bF}{{\mathbf F}}
\newcommand{\Xsm}{X_{\text{sm}}}

\usepackage{fullpage}

\title{Semiample Bertini theorems over finite fields }

\author{Daniel Erman}
\address{Department of Mathematics \\ University of Michigan \\
  Ann Arbor, MI 48109} 
\email{erman@umich.edu}

\author{Melanie Matchett Wood}
\address{Department of Mathematics \\ University of Wisconsin-Madison\\
  Madison, WI, 53706 \\ and American Institute of Mathematics\\360 Portage Ave \\
Palo Alto, CA 94306}
\email{mmwood@math.wisc.edu}

\begin{document}

\begin{abstract}
We prove a semiample generalization of Poonen's Bertini Theorem over a finite field that implies the existence of smooth sections for wide new classes of divisors.  The probability of smoothness is computed as a product of local probabilities taken over the fibers of the morphism determined by the relevant divisor.  
We give several applications including a negative answer to a
question of Baker and Poonen by constructing a variety (in fact one of each dimension) which provides
a counterexample to Bertini over finite fields in arbitrarily large projective
spaces. As another application, we determine the probability of smoothness for
curves in Hirzebruch surfaces, and the distribution of points on those smooth
curves.
\end{abstract}
\maketitle

\section{Introduction}
Fix a finite field $\F_q$.  What is the probability that a projective plane curve is smooth as the degree gets large?
More precisely, we let $C_d$ be a (uniform) random degree $d$ curve over $\F_q$ in $\P^2$, and we ask for the limit (if it exists) of the probability that $C_d$ is smooth.
 Assuming the local probability of being smooth at a point acts independently over the points of $\P^2$,
 we would predict that
\begin{equation}\label{E:predict}
 \lim_{d\ra \infty} \Prob(C_d \textrm{ is smooth})=\prod_{P\in \P^2} \lim_{d\ra\infty}\Prob(C\textrm{ smooth at }P) 
\end{equation}
This heuristic prediction is proven by Poonen \cite[Theorem~1.1]{poonen}, which
treats a much more general scenario and gave the first Bertini Theorem over a finite field. (See Remark 2 after Theorem~\ref{T:bertini}).

Suppose now we take curves in $\P^1\times\P^1$, which have a bidegree $(n,d)$.  There are many ways to let the ``degree'' of the curve go to infinity that are not treated by Poonen's theorem.  For instance, what about curves of bidegree $(2,d)$ as $d\ra\infty$?

The above heuristic does not apply to such curves, and to see this we consider the first projection $\pi: \P^1\times\P^1 \ra \P^1$.
If $C$ is a curve of bidegree $(2,d)$, then a fiber $\pi^{-1}(P)$ either is a component of $C$ or intersects $C$ with in zero-dimensional scheme of multiplicity $2$.  
Now if
$\pi^{-1}(P)$ contains more than two singular points of $C$, then every point of $\pi^{-1}(P)$ must be a singular point of $C$ (as that 
fiber must be a non-reduced component of $C$).  Thus, singularity at some points can determine singularity at other points in the fiber, and we conclude that the local smoothness probabilities fail to behave independently within a single fiber, even as $d\to \infty$.
However, we will show in Theorem~\ref{T:bertini} that this is the only dependence as $d\to \infty$; namely, assuming independence among the fibers of $\pi$ predicts the correct answer.

For a third example, let $\pi\colon X\to \P^2$ be the blow-up at an $\F_q$ point, with exceptional divisor $E$.  Curves in $X$ have a bidegree $(d,n)$ where $n+d$ is the degree of the image in $\P^2$ and $n$ is the intersection number with $E$. 
If we take curves with $n$ fixed and $d\to \infty$, then for a similar reason to the failure of independence along $\pi^{-1}(P)$ above, smoothness fails to act independently for points along $E$.  This turns out to be the only dependence; as $d\to \infty$, assuming independence among the fibers of $\pi$ predicts the correct probability of smoothness for these curves.

These examples 
motivate our main result, stated below and proven at the end of \S\ref{S:main theorem}.
When we say $D$ is a \emph{random divisor in $|nA+dE|$ as $d\ra\infty$}, we define (if the limit exists)
$$
\Prob(D\in \mathcal{P}):= \lim_{d\ra\infty} \Prob(D_d\in\mathcal{P})
$$
where $D_d$ is a uniform random divisor in $|nA+dE|$ defined over $\F_q$ and $\mathcal P$ is any set.
We take the convention that $D$ is smooth at any point it does not contain.

\begin{theorem}[Semiample Bertini]\label{T:bertini}
Let $X$ be a smooth projective variety over $\F_q$, with a very ample divisor $A$ and a globally generated divisor $E$.  
Let $\pi$ be the map given by the complete linear series on $E$.
\[
 \pi\colon X \overset{|E|}{\longrightarrow} \P^M
\]
There exists an $n_0$, depending only on
$\dim X$ and $\operatorname{char}(\F_q)$, such that for $n\geq n_0$,
the probability of smoothness for 
a random 
$D\in|nA+dE|$ as $d\ra\infty$ is given by the product of local probabilities taken over the fibers of $\pi$:
\[
\Prob(D\textrm{ is smooth})
=\prod_{P\in  \P^M} \Prob(D \textrm{ is smooth at all points of }\pi^{-1}(P)).
\]
The product on the right converges, is zero only if some factor is zero, and is always non-zero for $n$ sufficiently large.
\end{theorem}

\paragraph{{\bf 1.}}  When $E$ is not very ample, each fiber of $\pi$ may consist of many points and the fibers  may have different dimensions.  As in the examples considered above, singularity at points of a single fiber of $\pi$ will generally be dependent but the theorem shows that this is the only dependence as $d\to \infty$.

 \smallskip

\paragraph{{\bf 2.}}  
Our result specializes to Poonen's Bertini Theorem \cite[Theorem~1.1]{poonen} in the case where $A=E$.  One minor difference in that case is that Poonen works
with sections defined on the ambient projective space whereas we work with sections defined on $X$. However, for $d\gg 0$, there is a natural surjection $H^0(\P^M,\cO(d))\to H^0(X,\cO(dA))$, and hence this difference does not affect the asymptotics.  

At the opposite extreme, if $E=\cO_X$ then the equality is trivial, and the statement that the product is non-zero for $n\gg 0$ follows from Poonen's Bertini Theorem.

 \smallskip
 
\paragraph{{\bf 3.}} 
In characteristic $0$, the Bertini theorem for semiample divisors implies that a general section $D\in |dE|$ is smooth~\cite[Theoreme~6.10]{jouanolou}.  This can fail in positive characteristic, even over an algebraically closed field, without adding enough additional ampleness (or, alternately, adding separability hypotheses).  For instance, in characteristic $3$, every geometric vertical fiber of the smooth quasi-elliptic surface
$
V(y^2zs=x^3s-tz^3) \subseteq \P^2 \times \P^1,
$
is singular.\footnote{We thank Laurent Moret--Bailly for this example, which appeared on {\tt mathoverflow.net}~\cite{lmb-mo}.}  
 
 \smallskip
 
\paragraph{{\bf 4.}}  
If $n$ is not sufficiently large, then some factors in Theorem~\ref{T:bertini} may equal zero.  In fact, for any $m>0$, we construct a smooth variety $X$ of dimension $m$ that provides counterexamples to Bertini under (nondegenerate) embeddings into arbitrarily large projective spaces.  This yields a negative answer to a question posed by Baker and
given in \cite[Question~4.1]{poonen}.  See \S\ref{subsec:anti bertini}.
 \smallskip
 
\paragraph{{\bf 5.}}
Although for a given ${P\in  \P^M}$, the local probability $$\Prob(D \textrm{ is smooth at all points of }\pi^{-1}(P))$$ is a priori a limit in $d$,
this limit in fact stabilizes for sufficiently large $d$ (Lemma~\ref{lem:independence} \eqref{item:stabnod})
as long as $\dim \pi(X)>0$.  More concretely, if $P^{(2)}$ is the first-order infinitesimal neighborhood of a point $P$, then for \emph{any} $d$ that is at least
the Castelnuovo--Mumford regularity of  the sheaf $\pi_*(\cO_X(nA))\otimes \cO_{P^{(2)}}$ on $\P^M$, we have the equality:
\[
1-\Prob\left(\begin{matrix}D \textrm{ is smooth at } \\ \textrm{ all points of }\pi^{-1}(P)\end{matrix}\right)
=
\frac{
\#\left\{\begin{matrix}f\in H^0(X,\cO_{X}(nA+dE)) \textrm{ where } \\  f|_{Q^{(2)}}=0 \text{ for some } Q\in\pi^{-1}(P) \end{matrix}\right\}
}{\#H^0(X,\cO_{X}(nA+dE))}
\]
We give sample computations of these probabilities in Example~\ref{E:smoothinP1xP1}.

\smallskip
 
\paragraph{{\bf 6.}}
Our result depends on the choice of both $E$ (which determines the fibers), and the choices of $A$ and $n$ (which affect the local probability at a given fiber).  Poonen's Bertini Theorem did not depend on the choice of embedding $X\subseteq \P^N$, and the analogue here is that Theorem~\ref{T:bertini} depends only on the image $\pi(X)$ of the map determined by $E$.

\smallskip

\paragraph{{\bf 7.}}  Our proof yields an explicit bound for $n_0$.  For instance, we may set
\[
n_0=\max((\dim X+1)\cdot \dim \pi(X) - 1, (\dim \pi(X))\cdot \operatorname{char}(\F_q)+1).
\]
See~\eqref{eqn:choosing n0} in the proof of Lemma~\ref{L:med}.  Better bounds exist in some cases, as discussed in~\S\ref{S:better bounds}.

\medskip

We prove Theorem~\ref{T:bertini} in \S\ref{S:main theorem} as a special case of the more versatile Theorem~\ref{T:taylor} that allows for prescribed
behavior at finitely many closed points.
Although we follow the broad outline of Poonen's sieving proof of \cite[Theorem~1.2]{poonen}, our proof differs in essential ways.  
Poonen separates the points of $X$ into low, medium, and high degree points,
and his analysis of each case relies crucially on ampleness to obtain surjectivity of a map of sections.
For instance, Lemma 2.1 of \cite{poonen} is the key to the low and medium degree arguments.  The high degree case is the most difficult
part of Poonen's proof, and that argument relies on ampleness (for~\cite[Lemma~2.5]{poonen}), plus a very clever use of characteristic $p$ to decouple the vanishing sets of a function and its derivatives.

Since we have only semiampleness, we cannot rely on surjectivity, and it is thus a more subtle problem to describe or estimate the image of the analogous map
(see Lemmas~\ref{lem:independence} and \ref{L:sizeofimage}).
Further, our analysis separates the points of $\pi(X)$
(instead of $X$) into low, medium, and high degrees, and thus new possibilities arise such as the existence of high degree points in $X$ whose image in $\pi(X)$ is a medium degree point. 
 This significantly complicates the medium degree argument, which was one of the simpler parts of 
\cite{poonen} but is the hardest part of our proof.  For our medium and high degree arguments, we extend Poonen's decoupling idea to a bigraded (and semiample) setting, but in our situation the decoupling must be combined with several new ideas to provide sufficiently good bounds.
For example, we have to keep track of separate bounds for points with each degree image under $\pi$ in order
to balance the large number of points and their small chance of being singular points.  (See also the discussion
before Lemma~\ref{L:2kindshigh}.)
We then apply the resulting estimates in three different ways in our medium degree, high degree, and convergence arguments.

Poonen's Bertini Theorem has been generalized and applied in several settings.
  Some of these generalizations or other related results include~\cite{bucur-kedlaya,nguyen,poonen-sieve,poonen-given-sub,vakil-wood}.
 Several recent results have applied Poonen's Bertini Theorem with Taylor coefficients~\cite[Theorem~1.2]{poonen} to compute the distribution of the number of ($\F_q$) points on families of smooth varieties.  For instance,~\cite{bdfl-planar} uses an effective version of this result for $\P^2$ to compute the distribution of the number of points on a smooth plane curve of degree $d$ as $d\to \infty$.   In addition, Bucur and Kedlaya generalize \cite[Theorem~1.2]{poonen} to the case of complete intersections~\cite[Theorem~1.2]{bucur-kedlaya}, and they use this to compute the distribution of the number of points on  smooth complete intersections in $\P^n$ (as a limit as the defining degrees of the complete intersection go to $\infty$)~\cite[Corollary~1.3]{bucur-kedlaya}.  There has been a lot of other recent activity related to asymptotic point counting on curves over a fixed finite field, where the genus goes to $\infty$~\cite{bezerra-etc,bdfl-trigonal,bdfl-pfold,elkies-etc,kresh-etc,kurlberg-rudnick,Kurlberg-Wigman-Gaussian,li-maharaj,temkine,wood-trigonal}.

In a similar vein, we apply Theorem~\ref{T:taylor} to compute the distribution of the number of points on smooth
curves in several bidegree families on Hirzebruch surfaces (including $\P^1\times\P^1$) as $d\ra \infty$ (Theorem~\ref{T:Hpoints}).  For example,
 as $d\to \infty$, we show in Corollary~\ref{cor:avg 3d} that the average number of points on a smooth curve of type 
$(2,d)$ in $\P^1\times \P^1$ is $q+2+\frac{1}{q^3+q^2-1}$.

Along the way, in Theorem~\ref{T:Hsmooth}, we compute the probability that a curve of type $(n,d)$ in a Hirzebruch surface is smooth, for any fixed $n\geq 1$, as $d\ra\infty$.
In the case $n=3$, this counts smooth trigonal
curves in a \emph{single} Hirzebruch surface, as compared the theorem
of Datskovsky and Wright \cite{datskovsky-wright}, which counts all smooth trigonal curves,
that is, smooth trigonal curves in \emph{all} Hirzebruch surfaces, by
counting the function fields of the curves using an adelization of
Shintani zeta functions.
Y. Zhao~\cite{zhao} has proven a
new error bound for the Datskovsky-Wright theorem using geometric
methods, and in particular by relating singular curves in one
Hirzebruch surface to smooth curves in another.

\subsection{Outline of paper}

This paper is organized as follows. \S\ref{S:notation} outlines the basic notation and setup. \S\ref{S:main theorem} states Theorem~\ref{T:taylor}, which is a variant of Theorem~\ref{T:bertini} that allows for prescribed behavior at a finite number of points; we also state several overarching lemmas on low/medium/high degree singularities and we show that these lemmas imply Theorems~\ref{T:bertini} and~\ref{T:taylor}.  \S\ref{S:examples} contains example applications of the result.  \S\ref{S:setup} and \S\ref{S:lmh} contain the technical heart of the paper, as we prove a collection of lemmas that yield asymptotic independence among the fibers of $\pi$ and provide control over the behavior of the low/medium/high degree singularities, eventually proving the lemmas stated in \S\ref{S:main theorem}.  In \S\ref{S:convergence} we prove the convergence of the product and in \S\ref{S:better bounds} we discuss special cases where we can improve the bound on $n_0$.  Finally, \S\ref{S:applications} contains a number of further applications of the result including the point counting applications mentioned above.

\section*{Acknowledgements}
We thank Alina Bucur, Kiran Kedlaya, Rob Lazarsfeld, Bjorn Poonen, Kevin Tucker, and Ravi Vakil for useful conversations.
The first author was supported by a National Science Foundation fellowship and by a Simons Foundation Fellowship.
The second author was supported by an American Institute of Mathematics Five-Year Fellowship and National Science Foundation grant DMS-1147782.

\section{Notation}\label{S:notation}
Let $X$ be a projective variety of dimension $m$ over $\F_q$, with a very ample divisor $A$ and a globally generated divisor $E$.  We follow the convention that a variety over $\F_q$ is integral over $\F_q$, but not necessarily geometrically integral.
We let $p=\operatorname{char}(\F_q)$ and consider the maps
\[
\iota\colon X \overset{|A|}{\longrightarrow}\P^N \quad \text{ and } \quad \pi\colon X \overset{|E|}{\longrightarrow} \P^M
\]
given by the complete linear series on $A$ and  $E$, respectively.  We let $B=\pi(X)$ and $b=\dim B$.  
Since $A$ is very ample, we have $X\subseteq \P^N$ and also $X \subseteq \P^N \times \P^M$. 
For subvarieties $Y\subseteq X$, we use $\deg_A Y$ to denote the degree of $Y\subseteq \P^N$.

If $\dim\pi(X)>0$ and if $W$ is a $0$-dimensional subscheme of $\P^M$, then $\cO_{\pi^{-1}(W)}(E)$ is trivial on each 
component of $\pi^{-1}(W)$ and thus on $\pi^{-1}(W)$.
While there is no natural choice of trivialization, we 
can pick one, e.g. by dividing by a non-vanishing $\P^M$ coordinate for each component of $W$, and
we do this without further remark.  

We use $O$ notation in the paper, and the constant in the $O$ notation is a function of $X,E,A$.
For two functions $f_{X,E,A}, g_{X,E,A} : \mathbb{N}^s \ra \mathbb{R}$, we write $f= O(g)$ if there exists some positive real $c_{X,E,A}$ such that $f_{X,E,A}\leq c_{X,E,A} g_{X,E,A}$ for all values in $\mathbb{N}^s$.

We let $R_{n,d}:=H^0(X,\cO_X(nA+dE))$.  For a section $f\in R_{n,d}$ we use the notation $H_f$ for the corresponding divisor in $|nA+dE|$. 
  We let $S_{n,d}\sub \F_q[s_1,\dots,s_N,t_1,\dots,t_M]$ be the polynomials of degree
at most $n$ in the $s_i$ and at most $d$ in the $t_i$.  We will take affines $\A^N\sub \P^N$ and $\A^M\sub \P^M$, and use the
$s_i$ and the $t_i$ as their coordinates, respectively.  We may identify $S_{n,d}$ with $H^0(\P^N\times \P^M, \cO(n,d))$ in the natural way.

Fix a closed point $P\in B$.
If $f\in R_{n,d}\setminus\{0\}$, then $H_f$ is smooth at all points of $\pi^{-1}(P)$ 
if and only if $f$ does not vanish on any first order infinitesimal neighborhood of a closed point of $\pi^{-1}(P)$.  We thus introduce the following notation: for a point $Q$ in a scheme $Y$, we will use the notation $Q^{(2)}$ to denote the first order infinitesimal neighborhood.  For instance, with $P\in B$, $P^{(2)}:= \Spec(\cO_{B,P}/\mathfrak m_P^2)$.  More generally, if $W\subseteq Y$ is a smooth zero-dimensional scheme with irreducible components $P_1, \dots, P_s$, then we set $W^{(2)}=\cup_{i=1}^s P_i^{(2)}$.
We also set $X_{W^{(2)}}:=X\times_B W^{(2)}$ and $X_W:=X\times_B W$.

We will sometimes discuss the (Castelnuovo--Mumford) regularity of a sheaf $\mathcal F$ on $X$ with respect to the very ample bundle $A$.  We denote this as $\reg_A \mathcal F$.  The main property that we will use about regularity is that it provides an effective bound on Serre vanishing, i.e. $H^i(X,\mathcal F(nA))=0$ for all $i>0$ and $n\geq \reg_A \mathcal F-1$.  See~\cite[\S4]{eisenbud-syzygies} or \cite[\S1.8]{lazarsfeld-positivity} for background.

\section{Main theorem}\label{S:main theorem}
We will prove Theorem~\ref{T:bertini} as a special case (where $Z=\emptyset$) of Theorem~\ref{T:taylor} given below.  Like Poonen's~\cite[Theorem~1.2]{poonen}, this version is more versatile for applications.  For instance, see \S\ref{subsec:point counting} below for applications to point counting for smooth curves on Hirzebruch surfaces.

 When we say that $f$ is a \defi{random section of $R_{n,d}$ as $d\to \infty$}, for any set $\mathcal P$, we define
\[
\Prob(f\in \mathcal P):=\lim_{d\to \infty} \Prob(f_d\in \mathcal P)
\]
where $f_d$ is a uniform random section in $R_{n,d}$.

\begin{theorem}[Semiample Bertini with fiberwise Taylor coefficients]\label{T:taylor}
Let $X$ be a projective variety over $\F_q$, with a very ample divisor $A$ and a globally generated divisor $E$.  Let $Z\subsetneq \pi(X)$ be a finite subscheme of $\pi(X)$.
  Assume that $X \setminus \pi^{-1}(Z)$ is smooth and let $n_{0}:=\max(b(m+1)-1,bp+1)$, with the constants as defined in \S\ref{S:notation}.

For all $n\geq n_0$ and for all $T\subseteq H^0(\pi^{-1}(Z),\cO_{\pi^{-1}(Z)}(nA))$, and for a random  section $f\in R_{n,d}$ as $d\ra\infty$, we have
\[
\Prob\left(\begin{matrix} H_f \cap (X \setminus \pi^{-1}(Z)) \text{ is} \\ \text{ smooth and } f|_{\pi^{-1}(Z)}\in T\end{matrix} \right)=
\Prob(f|_{\pi^{-1}(Z)}\in T)
 \prod_{P\in \pi(X)\setminus Z} \Prob\left(\begin{matrix} H_f \text{ is smooth at}\\ \text{ all points of } \pi^{-1}(P)\end{matrix}\right).
\]
The product over $P\in \pi(X)\setminus Z$ converges, is zero only if some factor is zero, and is always non-zero for $n$ sufficiently large.
\end{theorem}

In Remark 5 after Theorem~\ref{T:bertini}, we noted that the factors in the product on the right stabilize for $d\gg 0$.
The same is true for $\Prob(f|_{\pi^{-1}(Z)}\in T)$. 
More precisely, 
the image of
$\phi_{d}\colon H^0(X,\cO_X(nA+dE))\to H^0(\pi^{-1}(Z),\cO_{\pi^{-1}(Z)}(nA))$ stabilizes for $d\gg  0$ 
to an image $I$ which is given in Lemma~\ref{lem:independence}, and the factor $\Prob(f|_{\pi^{-1}(Z)}\in T)$ then equals the fraction $\frac{\#(T\cap I)}{\# I}.$  More explicit descriptions of how to compute the local factors appear in \S\ref{subsec:sections from PNPM}.

The following lemmas (proven in Section~\ref{S:lmh}) control the probability of singularity in fibers $\pi^{-1}(P)$ for $P$ of low, medium, and high
degree, respectively.  As in the proof of the main results of~\cite{poonen}, the key to our sieving argument is to prove that singularities of medium or high degree occur so infrequently that they are irrelevant to the asymptotics.
 For some $e_0\geq 1$ (which will eventually go to $\infty$), we define
\[
\mathcal P^{\text{low}}_{e_0,n,T}:=\bigcup_d \left\{ \begin{matrix} f\in R_{n,d} \text{ where } H_f \text{ is smooth at all points } \\ Q\in X \setminus \pi^{-1}(Z) \text{ with } \deg(\pi(Q))< e_0, \text{ and } f|_{\pi^{-1}(Z)}\in T
\end{matrix} \right\},
\]
\[
\mathcal Q^{\text{med}}_{e_0,n}:=
\bigcup_d \left\{ \begin{matrix} f\in R_{n,d} \text{ where } H_f \text{ is singular at some } \\ Q\in X \setminus \pi^{-1}(Z) \text{ with } \deg(\pi(Q))\in [e_0,\tfrac{d}{\max(M+1,p)}]
\end{matrix} \right\},
\]
and
$$
 \calQhigh_{n} :=
 \bigcup_d \left\{ \begin{matrix} f\in R_{n,d} \text{ where } H_f \text{ is singular at some } \\ Q\in X \setminus \pi^{-1}(Z) \text{ with } \deg(\pi(Q))>\tfrac{d}{\max(M+1,p)}
\end{matrix} \right\}.
$$

\begin{lemma}\label{L:low approx}
For any $n$, we have
\[
\Prob(f\in \mathcal P^{\text{low}}_{e_0,n,T})
=
\Prob(f|_{\pi^{-1}(Z)}\in T)
\prod_{\substack{P\in B \setminus Z\\ \deg(P)< e_0}} \Prob\left(\begin{matrix} H_f \text{ is smooth at}\\ \text{ all points of } \pi^{-1}(P)\end{matrix}\right).
\]
\end{lemma}

\begin{lemma}\label{L:med}
For any $n\geq n_{0}$ (as defined in Theorem~\ref{T:taylor}),
 we have
\[
\lim_{e_0\ra\infty} \Prob\left(f\in \mathcal Q^{\text{med}}_{e_0,n}\right) = 0.
\]
\end{lemma}

\begin{lemma}
\label{L:high}
For any $n\geq 1$ we have 
 $$
\Prob\left( f\in \calQhigh_{n}\right)=0.
$$
\end{lemma}

\begin{proof}[Proof of Theorem~\ref{T:taylor} from Lemmas~\ref{L:low approx}--\ref{L:high} and Propositions~\ref{P:converges} and \ref{P:notzero}]
For any $n$ and any $e_0$, we have that
\begin{align*}
\Prob(f\in \mathcal P^{\text{low}}_{e_0,n,T})
&\geq \Prob\left(\begin{matrix} H_f \cap (X \setminus \pi^{-1}(Z)) \text{ is} \\ \text{ smooth and } f|_{\pi^{-1}(Z)}\in T\end{matrix} \right)\\
&\geq \Prob(f\in \mathcal P^{\text{low}}_{e_0,n,T})-\Prob(f\in \mathcal Q^{\text{med}}_{e_0,n})-\Prob(f\in\mathcal Q^{\text{high}}_{n}).
\end{align*}
 Then for any $n\geq n_0$, applying Lemmas~\ref{L:low approx}, \ref{L:med}, and \ref{L:high} and taking the limit as
$e_0\ra\infty$ yields Theorem~\ref{T:taylor}.
The convergence and non-zero claims are proven in Propositions~\ref{P:converges} and \ref{P:notzero}.
\end{proof}

\begin{proof}[Proof of Theorem~\ref{T:bertini} from Theorem~\ref{T:taylor}]
If $\dim \pi(X)=0$ then the equality of probabilities is tautological; further, the existence of an $n_0$ giving a non-zero product follows from the corresponding statement in Theorem~\ref{T:taylor}.

So we assume $\dim \pi(X)\geq 1$.  Choose $Z=\emptyset$.  We then have
\[
\Prob\left(H_f \cap X \text{ is} \\ \text{ smooth}\right)=
 \prod_{P\in B} \Prob\left(\begin{matrix} H_f \text{ is smooth at}\\ \text{ all points of } \pi^{-1}(P)\end{matrix}\right).
\]
The linear series $|nA+dE|$ is the projectivization of the vector space $R_{n,d}$.  We have $\dim R_{n,d} = h^0(\pi(X),\pi_*(\cO_X(nA))\otimes \cO_{\P^M}(d))$, and since $\pi_*(\cO_X(nA))$ is torsion free and $\dim \pi(X)\geq 1$, the dimension of this vector space goes to $\infty$ as $d\to \infty$.  Hence the probability that $f=0$ goes to $0$ as $d\to \infty$.  It thus makes no difference in the asymptotic probabilities whether we consider divisors in the linear series or sections from $R_{n,d}$.
\end{proof}

\section{Applications}\label{S:examples}

In this section, we give a few examples and applications of our main result.

\begin{example}[Bigraded Bertini]\label{E:multigraded}
If $X\subseteq \P^{i}\times \P^{j}$ is any smooth subvariety, then smooth hypersurface sections of $X$ exist in bidegree $(n,d)$ for any 
$d\gg n \gg 0$.  This follows from Theorem~\ref{T:bertini} where $A:=\cO_X(1,1)$ and $E:=\cO_X(0,1)$.
\end{example}

\begin{example}[Bigraded Anti-Bertini]\label{E:multi anti}
Given any $N$, there exists a smooth, geometrically integral hypersurface $X\subseteq \P^{i}\times \P^{j}$ where all hypersurface sections of $X$ of bidegree $(n,d)$ are singular for $n\leq N$ and all $d\geq 0$.  
One application of this example is to provide a negative answer to Baker's open question \cite[Question 4.1]{poonen}.  See \S\ref{subsec:anti bertini}
 for details.
\end{example}

\begin{example}[$E$ very ample]\label{E:very ample}
If $E$ is very ample then the product of local probabilities in Theorem~\ref{T:bertini} equals $\zeta_X(m+1)^{-1}$, as  in~\cite[Theorem 1.1]{poonen}.
This is because for a closed point $Q\in X$ and $d$ sufficiently large, $H^0(X,\cO(nA+dE))\ra H^0(Q^{(2)},\cO_{Q^{(2)}})$ is surjective
and so the probability of singularity at $Q$ is $q^{-(m+1)\deg(Q)}$.
For instance, on $\P^i\times \P^j$, the probability that a hypersurface of bidegree $(d,d+k)$
 is smooth, for fixed $k$ and $d\ra\infty$, equals $\zeta_{\P^i\times \P^j}(i+j+1)^{-1}$.

 In addition, when $E$ is very ample (or even if $\pi$ is a isomorphism away from a finite number of positive dimensional fibers)  we can choose $n_0=1$ in Theorem~\ref{T:bertini} (see Proposition~\ref{P:n0is1}).  The key observation is that
the positive dimensional fibers can be subsumed into the low degree argument, and then the medium degree fibers 
can be dealt with in a manner similar to \cite[Lemma 2.4]{poonen}.
\end{example}

\begin{example}[Smooth curves in $\P^1\times\P^1$]\label{E:smoothinP1xP1}
The following table uses Theorem~\ref{T:bertini} to give the probability of smoothness for curves of various bidegress in $\P^1\times \P^1$.

\begin{center}
\renewcommand{\arraystretch}{1.5}
\begin{tabular}{ | c | c | } \hline
Bidegree & Smoothness Probability (as $d\to \infty$)\\ \hline
 $(13d+11,7d)$ & $(1-q^{-1})(1-q^{-2})^2(1-q^{-3})$ \\  \hline
 $(2,d)$ & $\prod_{P\in\P^1}\left(1-q^{-2\deg(P)}-q^{-3\deg(P)}+q^{-4\deg(P)}\right)$ \\  \hline
 $(9,d)$ & $(1-q^{-1})(1-q^{-2})^2(1-q^{-3})$ \\  \hline
\end{tabular}
\end{center}
The first row uses Example~\ref{E:very ample} and would be the same for any sequence of the form $nA+dE$ where $E$ is ample.  The second and third rows rely on computations for
local probabilities that appear in~\S\ref{subsec:point counting},
and the third row would be the same for $9$ replaced by any $n\geq 3$.  
The equality of the first and third probabilities is somewhat of a coincidence,
 and we do not expect that the second probability is a rational number.
 For example, if $q=2$ we have  $(1-q^{-1})(1-q^{-2})^2(1-q^{-3})=63/256$ and 
$\prod_{P\in\P^1}(1-q^{-2\deg(P)}-q^{-3\deg(P)}+q^{-4\deg(P)})\approx 0.2839863\dots$
\end{example}

\begin{example}[Pointless $n$-gonal curves]
We apply Theorem~\ref{T:bertini} with $X=\P^1\times \P^1$, $A=\cO(n,1)$, $E=\cO(0,1)$,
$Z$ the union of the $\F_q$ points in $\P^1$, and $T$ specifying that the curves should be smooth
at all points of $\pi^{-1}(Z)$ but not contain any degree $1$ points.  For any $n\geq 2$, this produces smooth $n$-gonal curves of arbitrarily large genus with no $\F_q$ points.  For $n=3$, combining this with a similar application of Theorem~\ref{T:bertini} on a Hirzebruch surface, we prove the existence of trigonal curves of genus $g$ with no $\F_q$ points for all $g\gg 0$.  (See the recent work \cite{stichtenoth} which proves there is some pointless curve of every sufficiently large genus,
and \cite{becker-glass} which proves there is a pointless hyperelliptic curve of every sufficiently large genus, as well as the related work \cite{howe-etc,stark}.)
\end{example}

\section{Lemmas}\label{S:setup}

In this section, we prove several lemmas which will be used to prove Lemmas~\ref{L:low approx}--\ref{L:high} in Section~\ref{S:lmh}.  The following bound on degrees will be needed in applications of B\'ezout's theorem and the Lang--Weil bound in later lemmas.

\begin{lemma}\label{L:Global1}
Let $P$ be a closed point of $B$.
There exists $d_1$ (a function of $X,A,E$) such that $d_1\deg(P)$ is at least the sum of the $A$-degrees of all of the irreducible components of $X_P$ (note that $X_P$ may fail to be equidimensional), for any $P\in B$.
\end{lemma}
\begin{proof}
By choosing a flattening stratification for $\pi$, we see that the geometric fibers of $\pi$ attain only finitely many Hilbert polynomials.  
The existence of a Gotzmann bound for each of these finitely many Hilbert polynomials leads to a global bound on the Castenuovo--Mumford regularity of the geometric fibers of $\pi$ (i.e. a global bound on $\reg_A \mathcal I_{X_P}$ for geometric points $P$ in $B\otimes_{\F_q} \overline{\F_q}$) which implies the existence of a global bound on the degrees of the generators of the defining ideals of the geometric fibers.  
By~\cite[Prop.~3.5]{bayer-mumford}, this in turn implies the existence of a global bound $c_r$ on the degree of the union of the $r$-dimensional components of any geometric fiber.  We may set $d_1:=\sum_{r=0}^m c_r$.
\end{proof}

At the start of his sieving argument, Poonen used ampleness to obtain a surjectivity result \cite[Lemma~2.1]{poonen} to interpolate at any finite collection of points.  Since we are working with a semiample divisor $E$ instead of a very ample divisor, we may lose surjectivity from sections on $X$ to sections on the fibers $X_P$ as well as to sections on points within those fibers.  Nevertheless, the following provides asymptotic independence statements for the fibers of $\pi$, and it will form the start of our sieving argument.
\begin{lemma}\label{lem:independence}
Assume that $\dim \pi(X)>0$ and hence that $M>0$.  Let $W\subseteq \P^M$ be any $0$-dimensional subscheme.
Then, as $d\ra\infty:$
\begin{enumerate}
 \item \label{item:stabnod} 
The image of
\[
H^0(X,\cO_X(nA+dE))\to H^0(X_W, \cO_{X_W}(nA))
\]
stabilizes for sufficiently large $d$ to  the image of
$$
H^0(W, \pi_*(\cO(nA))\tensor \cO_W) \to H^0(X_W, \cO_{X_W}(nA)).
$$

\item \label{item:stabind} If the connected components of $W$ are $Q_1, \dots, Q_s$, then the stable image above equals the stable image of
\[
\bigoplus_{i=1}^s \im \left( H^0(X,\cO_X(nA+dE))\to H^0(X_{Q_i}, \cO_{X_{Q_i}}(nA)) \right).
\]

\end{enumerate}

\end{lemma}

\begin{proof}
For \eqref{item:stabnod}, 
we first observe that we have a commutative diagram that factors our map of interest:
\[
\xymatrix{
 H^0(X,\cO(nA+dE)) \ar[rr]^-g\ar[d]^{\cong}& &H^0(X_W, \cO_{X_W}(nA))\ar[dd]^\cong\\
 H^0(B,\pi_*(\cO_X(nA))(d))\ar[u]\ar[r]^-h&H^0(W, \pi_*(\cO_X(nA)) \tensor \cO_{W}(d) )\ar[d]^\cong&\\ 
 &H^0(W, \pi_*(\cO_X(nA)) \tensor \cO_{W} )\ar[r]^-k\ar[u]& H^0(W,\pi_*(\cO_{X_W}(nA))).\ar[uu].
 }
\]
Since $\pi_*(\cO(nA)) \ra \pi_*(\cO(nA)) \tesnor \cO_{W}$
is a surjection of coherent sheaves on $B$, and since $\cO_B(1)$ is very ample, 
it follows from Serre vanishing that for $d\gg 0$\footnote{It suffices to choose $d$ at least the Castelnuovo--Mumford regularity of $\pi_*(\cO_X(nA))\otimes \cO_W$ on $\P^M$.} the map $h$ is surjective.
Thus the image of $g$ equals the image of $k$, as claimed.

For \eqref{item:stabind}, the map
$$
H^0(W, \pi_*(\cO_X(nA))\tensor \cO_W) \to H^0(X_W, \cO_{X_W}(nA)).
$$
clearly decomposes as a direct product of the analogous maps on each component.
\end{proof}

We use $\pi_1$ and $\pi_2$ to refer to the projections from $\P^N\times \P^M$ onto the two factors. 
The following lemma provides a lower bound for size of the image of $S_{n,d}$ at a (possibly fattened)
point of $X$ or fiber of $\pi$, which we will use in several different ways to show that the singularity probabilities
in medium and high degree fibers are small enough asymptotically.

\begin{lemma}\label{L:sizeofimage}
Let $W\subseteq \A^N \times \A^M $ be a closed subscheme such that 
$\pi_1$ is an isomorphism on $W$ and
$\pi_2(W)$ is supported at a closed point $P$ of degree $e$.  Let $w:=\dim_{\F_q} H^0(W,\cO_W)$ (if $\dim W>0$ then we set $w=\infty$) and
 $r:=\deg(\pi_2(W))$.
Consider the restriction map
$$
\phi_{W,n,d}: S_{n,d} \ra H^0(W,\cO_W).
$$
Then, for $n,d\geq 0$, we have
\begin{itemize}
 \item $\# \im(\phi_{W,n,d}) \geq q^{\min(d+1, r)}$ and
\item for $d+1\geq r$, we have  $\# \im(\phi_{W,n,d}) \geq q^{\min(en+r, w)}.$
 \end{itemize}
\end{lemma}

\begin{proof}
 We  have the commutative diagram of $\F_q$-algebras,
\[
\xymatrix{
\F_q[s_1, \dots, s_N,t_1,\dots,t_M]\ar[r]&H^0(W,\cO_W)\\
\F_q[t_1, \dots, t_M]\ar[r]\ar[u]& H^0(\pi_2(W),\cO_{\pi_2(W)})\ar[u]
}
\]
where both vertical arrows are injections and both horizontal arrows are surjections.
We set $B_{n,i}$ to be the image of $S_{n,i}$ in $H^0(W,\cO_W)$.  We then obtain the diagram
\[
\xymatrix{
S_{n,i}\ar[r]&B_{n,i}\\
S_{0,i}\ar[r]\ar[u]&B_{0,i}.\ar[u]
}
\]
Since $B_{0,i+1}=\sum_j{t_jB_{0,i}} + B_{0,i}$, it follows that $B_{0,i}$ increases in dimension (as an $\F_q$ vector space) with each increase of $i$ until it stabilizes to $H^0(\pi_2(W),\cO_{\pi_2(W)})$.  
(This is the same idea used in \cite[Lemma 2.5]{poonen}.)
 Thus the first statement in the lemma follows.  

So for $i+1\geq r$, we have that $B_{0,i}=H^0(\pi_2(W),\cO_{\pi_2(W)}).$
 Since the natural map $S_{n,0}\otimes S_{0,i}\to S_{n,i}$ is an isomorphism for any $n$ and $i$, it follows that $B_{n,0}\otimes B_{0,i}\to B_{n,i}$ is surjective for any $n$ and $i$ and thus $B_{n,i}=B_{n,i+1}=\dots=B_{n,2i}$.  We observe that
\[
B_{n,i}\otimes B_{0,i}\to B_{n,2i}\cong B_{n,i}
\]
and hence $B_{n,i}$ has the structure of a $H^0(\pi_2(W),\cO_{\pi_2(W)})$-module for all $n\geq 0$ and any $i+1\geq r$.

The cardinality of any finite length $\cO_{B,P}$-module is a power of the cardinality of the residue field, which is $q^{\deg(P)}$.
So, for $i+1\geq r$, we have that $\# B_{n,i}$ is a power of $q^e$.
Now, since $B_{n+1,i}=\sum_j{s_jB_{n,i}} + B_{n,i}$, it follows that $B_{n,i}$ increases in dimension 
by at least $e$ (as an $\F_{q}$ vector space)
with each increase of $n$ until it stabilizes to $H^0(\pi_1(W),\cO_{\pi_1(W)})=H^0(W,\cO_{W})$. 
The second statement in the lemma follows.
\end{proof}

The following lemma will let us bound the singularity probability at points for which the bounds of Lemma~\ref{L:sizeofimage}
are insufficient, if we had to use those bounds for all the points.  
The starting point of this lemma is Poonen's idea for decoupling a function and its partial derivatives.  This idea was introduced in \cite[Lemma 2.6]{poonen}
to limit 
the number of potentially singular points by showing that, with high probability, the partial derivatives
have a zero-dimensional common vanishing locus.
However, our case requires a  more delicate analysis.  To begin with, there may be positive dimensional components on which the partial derivatives all vanish;
 we control some of these components using Lemma~\ref{L:sizeofimage} and determine that others do not contain points of interest.
We also use other new ideas, including how we control the actual bounds involved in this bigraded scenario, the distinguishing of the dimension of components and their images under $\pi$, and the distinguishing of 
the degree and $\pi$-relative degree of points.  When choosing our open sets in the proof, we follow~\cite[Proof of Lemma~2.6]{bucur-kedlaya}, as this can provide a more effective bound on the constants that arise.

We will apply this lemma in three contexts: 
to rule out (asymptotically almost surely) singularities at $Q\in X$ where $\pi(Q)$ has high degree;
to rule out (a.a.s.) singularities at $Q\in X$ of  high relative degree but where $\pi(Q)$ has medium degree;  and to control when  the product of Theorem~\ref{T:taylor} is zero. 
While the first context is analogous to where \cite[Lemma 2.6]{poonen} and \cite[Lemma~2.6]{bucur-kedlaya} are used,
the second context is new to our case and requires the most delicate bounds and argument.

\begin{lemma}\label{L:2kindshigh}
Let $\Xsm$ be the smooth locus of $X$ and 
let $j,J$ be integers.  Fix $n\geq 1$ and $d\geq 0$, and let $f\in R_{n,d}$ be chosen uniformly at random.
\begin{enumerate}
 \item The probability that $H_f$ has a singularity at a closed point $Q\in \Xsm$ with $\deg(\pi(Q))\in [j,\infty)$ is at most 
 \[
 O((n^m+d^m)q^{-\min(\lfloor d/p \rfloor+1,j)}).
 \] \label{E:overhigh}
\item  The probability that  $H_f$ has a singularity at a closed point $Q\in \Xsm$ with $\deg(\pi(Q))\in [j,\lfloor d/p \rfloor+1]$ and $\deg(Q)/\deg(\pi(Q))\geq J$
 is at most 
 $$
O((n^m+d^m) q^{-(\lfloor d/p \rfloor+1)}+  \sum_{e=j}^{\lfloor d/p \rfloor+1} e(n^m+e^m) q^{e(b-\min(J, \lfloor (n-1)/p \rfloor +1 ))}).\label{E:highovermed}
$$

\item Let $P$ be a closed point of $B$ with $\deg(P)=e$.  For $\lfloor d/p \rfloor+1\geq e$, 
the probability that $H_f$ has a singularity at a closed point $Q\in \Xsm$ with $\pi(Q)=P$ and  $\deg(Q)/e \geq J$
 is at most 
 $$
O((n^m+d^m) q^{-(\lfloor d/p \rfloor+1)}+   e(n^m+e^m) q^{-e\min(J, \lfloor (n-1)/p \rfloor +1 )}).
$$ 
\label{E:highoverP}

\end{enumerate}
\end{lemma}

\begin{proof}
We will prove all statements of the lemma simultaneously, occasionally dividing into cases.
Call a closed point $Q\in \Xsm$ \emph{admissible} in the first case if $\deg(\pi(Q))\in [j,\infty)$,  in the second case
if $\deg(\pi(Q))\in [j,\lfloor d/p \rfloor+1]$ and $\deg(Q)/\deg(\pi(Q))\geq J$, and in the third case if $\pi(Q)=P$
and   $\deg(Q)/e \geq J$.
 
Let $A:=\F_q[s_1,\dots,s_N]$.
We can cover $\Xsm$ with finitely many opens, so as to replace $\Xsm$ in the lemma by an open $U\subseteq \Xsm\cap(\A^N \times \A^M).$  
By again reducing to a finite cover, and by possibly relabelling the $s_i$, we can assume that $ds_1,\dots,ds_m$ generate the differentials in $U\cap \A^N$.
On $U$, we write $ds_k =\sum_{i=1}^m a_{k,i} ds_i$ for $k>m$, and $a_{k,i}\in \F_q(s_1,\dots,s_n)$ a regular function on $U$.  This gives maps $D_i : A\ra \F_q(s_1,\dots,s_n)$, so that on $U$, we have $df = \sum_{i=1}^m D_ifds_i$, and the
degree of the numerator of $D_if$ is $O(\deg(f))$.

For any closed point $Q \in U$, we have the maps
$$
S_{n,d}\cong H^0(\P^N\times\P^m,\cO(n,d)) \ra R_{n,d} \ra H^0(Q^{(2)},\cO_{Q^{(2)}}),
$$
so a randomly chosen element of $g\in S_{n,d}$ is at least as likely to give a $H_g\cap U$ singular at $Q$ as a randomly chosen element of 
$f\in R_{n,d}$ is to give a $H_f$ singular at $Q$.  Thus, we reduce to studying a randomly chosen $g\in S_{n,d}$.

Each coordinate function $t_j$ of $\A^M$ pulls back via $\pi$ to a regular function on $U$,
which, since $U\subseteq \A^N$, can be expressed as a polynomial $\phi(t_j) \in A$.
Thus, given $g\in S_{n,d}$, we may use $\phi$ to  write $g|_U$ as a polynomial
$\phi(g)$ in $s_1, \dots, s_N$.
For a given $g\in S_{n,d}$, a point $Q\in U$ is a singularity of $H_g\cap U$ if and only if $Q\in
\left\{\phi(g)=D_1\phi(g)  =\dots = D_m\phi(g)=0 \right\}\subseteq \A^N.$

 Following Poonen's decoupling idea~\cite[p.~1106]{poonen}, we will rewrite $g$ in such a way that the partial derivatives of $g|_U$ are largely independent.  We will choose $g\in S_{n,d}$ uniformly at random by choosing
$g_0\in S_{n,d}$, and $g_1,\dots,g_m \in  S_{\lfloor (n-1)/p \rfloor,\lfloor d/p \rfloor}$
and $h\in S_{\lfloor n/p \rfloor,\lfloor d/p \rfloor}$ uniformly at random, and then putting
$$
g=g_0+g_1^ps_1 +\dots+ g_m^p s_m +h^p.
$$
For $0\leq i\leq m$, we define 
$$
W_i=U \cap \left\{D_1\phi(g)  =\dots = D_i\phi(g)=0 \right\}\subset \A^{N}.
$$

\noindent{\em Claim 1:} For $0 \le i \le m-1$,
given a choice of $g_0,g_1,\dots,g_i$
for which each component $V$ of $W_i$ that contains an admissible point has $\dim(V) \le m-i$, 
the probability that some component $V'$ of $W_{i+1}$ contains
an admissible point and has
$\dim(V') > m-i-1$
is at most
\begin{itemize}
\item
$ O((n^m+d^m) \cdot q^{-\min(\lfloor d/p \rfloor+1,j  )})$
\item 
$ O((n^m+d^m) q^{-(\lfloor d/p \rfloor+1)}+  \sum_{e=j}^{\lfloor d/p \rfloor+1} e(n^i+e^i) q^{be-\min(eJ,e( \lfloor (n-1)/p \rfloor  +1))})$
\item 
$ O((n^m+d^m) q^{-(\lfloor d/p \rfloor+1)}+   e(n^i+e^i) q^{-\min(eJ,e( \lfloor (n-1)/p \rfloor  +1))})$
\end{itemize}
in the cases of the Lemma, respectively.

\bigskip
\noindent{\em Proof of Claim 1:}
Let $V_1$, \dots, $V_c$ denote the irreducible components of $(W_i)_\red$.
If all points of $V_k$ are non-admissible, then the same
will be true for components of $V_k\cap W_{i+1}$.
So, for $V_k$ containing an admissible point and with $\dim \pi(V_k)\leq m-i$, we will bound the probability $\mathcal{P}$ that
$\dim(V_k\cap W_{i+1})> m-i-1$, which happens exactly when
 $D_{i+1}\phi(g)$ vanishes on $V_k.$
Since $D_{i+i}\phi(g)=D_{i+i}\phi(g_0)+\phi(g_{i+1})^p$, given $g_0,\dots,g_i$, the choices for $g_{i+1}$
that give $D_{i+i}\phi(g)$ vanishing on $V_k$ are a coset of the functions vanishing on $V_k$ (or there are no such choices).
So we will give a lower bound for the size of the image $I:=\im(S_{\lfloor (n-1)/p \rfloor,\lfloor d/p \rfloor} \ra H^0(V_k,\cO))$ and thus a lower bound for $\mathcal{P}^{-1}=\# I$.

If $\dim\pi(V_k)\geq 1$, then there is a $t_i$ such that any polynomial in $t_i$ is non-vanishing on $V_k$, and so $\# I \geq q^{\lfloor d/p \rfloor+1}$ and
$\mathcal{P}\leq q^{-(\lfloor d/p \rfloor+1)}$. 
We may apply B\'ezout's theorem
to obtain $c = O(n^i+d^i)$.
For the contribution to Claim 1 from
$V_k$ with $\dim\pi(V_k)\geq 1$, we use the bound $\mathcal{P}\leq q^{-(\lfloor d/p \rfloor+1)}$ above for a single component, and $c = O(n^i+d^i)$.  

If $\dim\pi(V_k)= 0$ and $V_k$ contains an admissible point, then we apply Lemma~\ref{L:sizeofimage} with $W=V_k$, to obtain 
\begin{itemize}
 \item  $\#I\geq q^{\min(\lfloor d/p \rfloor+1,\deg(\pi(V_k))} \geq q^{\min(\lfloor d/p \rfloor+1,j)}$ in any case, and
\item in the second or third case of the lemma, $\# I\geq q^{   \min(  \deg(\pi(V_k)(\lfloor (n-1)/p \rfloor+1),  \dim H^0(V_k,\cO_{V_k}) ) }$.
\end{itemize}

 Since $i\leq m$, the first case of Claim 1 now follows, using the fact that $c = O(n^i+d^i)$.
For the second and third cases, we further determine an upper bound on the number of $V_k$ that can map to the same closed point $P=\pi(V_k)$ of degree $e$.
In $X_P$, any expression in the variables $t_i$ reduces to one of degree at most $e$, so there are at most
$O(e(n^i+e^i))$ of the  $V_k$ in $X_P$ (a priori the constant depends on the sum of the degrees of the irreducible components of  $X_P$, but this
is bounded by $d_1e$ by Lemma~\ref{L:Global1}).
The third case of Claim 1 now follows. 
 There are at most $2^b\deg(B)q^{be}= O(q^{be})$ points $P\in B$ of degree $e$ by the Lang--Weil bound \cite[Lemma 1]{LW}, and so
 from degree $e$ points the total contribution to $\mathcal{P}$ is
$$
O(q^{be}e  (n^i+e^i) q^{-\min(eJ,e (\lfloor (n-1)/p \rfloor +1) )}).
$$
Summing over $e=j, \dots, \lfloor d/p\rfloor +1$,  we obtain the contribution to the second case of Claim 1 coming from
$V_k$ with $\dim\pi(V_k)= 0$.  

\bigskip
\noindent{\em Claim 2:} Conditioned on a choice of $f_0,g_1,\dots,g_m$
for which each component $V$  of $W_m$ that contains an admissible point has  $\dim(V)=0$,
the probability that $H_g \cap U$ contains an admissible singular point is
\begin{itemize}
\item $ O((n^m+d^m)q^{-\min(\lfloor d/p \rfloor+1,j))}$ 
\item $ O(\sum_{e=j}^{\lfloor d/p \rfloor+1} q^{be} e(n^m+e^m)q^{-\min(eJ,e( \lfloor n/p \rfloor +1) )})
$ 
\item $ O( e(n^m+e^m)q^{-\min(eJ,e( \lfloor n/p \rfloor +1) )})
$ 
\end{itemize}
in the  cases of the Lemma, respectively.

\bigskip
\noindent{\em Proof of Claim 2:}
The singular points of $H_g \cap U$ are the points of $W_m$ on which $g$ vanishes.
By B\'ezout's theorem, there are $O(n^m+d^m)$ irreducible components of $W_m$, and there are
$O(e(n^m+e^m))$ irreducible components $V$  of $W_m$ with $\pi(V)$ a given point of degree $e$ (as in the proof of Claim 1).
Any component $V$ of $W_m$ that contains no admissible points will contain no admissible singular points.
For a given 
zero-dimensional admissible
component $Q$ of $W_m$,
the set of $h$ that give $g$ vanishing at $Q$ is either empty or a coset of the set of
$h\in S_{\lfloor n/p \rfloor,\lfloor d/p \rfloor}$ vanishing at $Q$.  Let 
$I:=\im(S_{\lfloor n/p \rfloor,\lfloor d/p \rfloor} \ra H^0(Q,\cO))$, and 
we apply Lemma~\ref{L:sizeofimage} with $W=Q$ to see that
\begin{itemize}
\item $\# I\geq q^{\min(\lfloor d/p \rfloor+1,j)}$ in any case
\item $\# I\geq q^{\min(\deg(\pi(Q))J,\deg(\pi(Q)) (\lfloor n/p \rfloor +1) )}$ in the second and third cases of the Lemma. 
\end{itemize}
The first and third cases of Claim~2 follow immediately, and the second case follows using that
there are $O(q^{be})$ points $P\in B$ of degree $e$.

The two claims combine to prove the lemma.
\end{proof}

\section{Low, medium, and high degree points}\label{S:lmh}

We now analyze the fibers of $\pi$ lying over low, medium, and high degree points of $B$ to prove Lemmas~\ref{L:low approx}--\ref{L:high} .

\subsection{Low degree points}

We prove Lemma~\ref{L:low approx} by applying Lemma~\ref{lem:independence}, which yields independent behavior among the various low degree fibers.  
\begin{proof}[Proof of Lemma~\ref{L:low approx}]
When $\dim\pi(X)=0$ the statement is trivial, and hence we assume $\dim \pi(X)>0$.
We apply Lemma~\ref{lem:independence} \eqref{item:stabind} in the case where $W$ is the disjoint union of $Z$ and $P^{(2)}$ for all points $P\in B\setminus Z$ with $\deg(P)< e_0$.
The factorization of the image given in Lemma~\ref{lem:independence} \eqref{item:stabind} implies 
Lemma~\ref{L:low approx}.
\end{proof}

\subsection{Medium degree points}

To prove Lemma~\ref{L:med}, which bounds the singularity probability in medium degree fibers, we will divide the points of those fibers further into those of low and high relative degree.  Lemma~\ref{L:fiberbound} will take care of the low relative degree points, and then we will apply Lemma~\ref{L:2kindshigh} to handle the high relative degree points.
We stratify $B$ into $\bigsqcup_k B_k$, where $B_k$ is the locus of points whose $\pi$-fiber has dimension $m-k$.
Note that $\dim(B_k)\leq k$, and $k\leq b$.    We also use $\Xsm$ to denote the smooth locus of $X$.

\begin{lemma}\label{L:fiberbound}
Let $P \in B_k$ with $\deg(P)=e$.
For $d\geq e(M+1)-1 $, we have that  
$$
\frac{\# \{ f\in R_{n,d} | H_f \text{ is singular at a point $Q\in X_P\cap \Xsm$ with $\frac{\deg({Q})}{e}\leq \frac{n+1}{m+1}$} \}}{\#R_{n,d}} = O(eq^{-e (k+1)}).
$$
\end{lemma}

\begin{proof}
Let $Q\in \Xsm$ be a relative degree $f$ point with $\pi(Q)=P$ (i.e. a point with residue field
$\F_{q^{ef}}$) and $f\leq \frac{n+1}{m+1}$.
The second case of Lemma~\ref{L:sizeofimage}, applied with $W=Q^{(2)}$, gives
 that the probability of singularity at $Q$ is $q^{-ef(m+1)}$ (as $e\leq \deg(\pi_2(Q^{(2)}))\leq e(M+1)$
and $e(n+1)\geq ef(m+1)=\dim H^0(Q^{(2)},\cO_{Q^{(2)}}).)$

We now add these up for all such points $Q$.
This is similar to the idea of Poonen's \cite[Lemma 2.4]{poonen} which adds up
singularity bounds for all medium degree points, but we are adding up over
low and medium \emph{relative} degree points in a fiber over a medium degree point in $B$.
Since the sum of the degrees of the irreducible components of $X_P$ is
at most $d_1\deg(P)$ (by Lemma~\ref{L:Global1}) and since the dimension of each component is at most $m-k$, the Lang--Weil bound \cite[Lemma 1]{LW} gives
\[
\#X_P(\F_{q^{ef}})\leq 2^{m-k}d_1eq^{ef(m-k)}  = O( eq^{ef(m-k)}).
\]
Thus, the proportion of sections we are bounding is at most
\begin{align*}
\sum_{f=1}^{\lfloor \frac{n+1}{m+1} \rfloor}  \#X_P(\F_{q^{ef}}) q^{-ef(m+1)} 
& =  \sum_{f=1}^{\lfloor (n+1)/(m+1) \rfloor} O\left( e q^{-ef(k+1)}\right)=O(e q^{-e(k+1)}).
\end{align*}
\end{proof}

We now  prove
 Lemma~\ref{L:med}.

\begin{proof}[Proof of Lemma~\ref{L:med}]
We always choose $e_0$ to be sufficiently large so that all points of $Z$ have degree  less than $e_0$.
Since $X \setminus \pi^{-1}(Z)$ is smooth, it follows that every point $Q$ with $\deg \pi(Q)\geq e_0$ is a smooth point of $X$.

We first consider singularities at a point $Q$ with $\frac{\deg(Q)}{\deg(\pi(Q))}\leq\frac{n+1}{m+1}$.
Lemma~\ref{L:fiberbound} implies that
\[
\frac{\# \displaystyle{\bigcup_{\substack{P\in B \\ e_0\leq \deg(P)\leq\frac{d+1}{M+1}}}}  \left\{ f\in R_{n,d} | H_f \text{  sing. at $Q\in X_P$
and $\tfrac{\deg(Q)}{\deg(\pi(Q))} \leq \tfrac{n+1}{m+1}$} \right\}}{\#R_{n,d}}
\]
is
\[
\sum_{k=0}^{b}  \sum_{e= e_0}^{\lfloor (d+1)/(M+1) \rfloor} \# B_k (\F_{q^e}) O(eq^{-e (k+1)}) =
\sum_{e= e_0}O(eq^{-e }).
\]
This follows by applying Lang-Weil~\cite[Lemma 1]{LW} to $B_k$ to obtain $\#B_k(\F_{q^e}) = O(q^{ek})$.
Since $\sum_{e\geq e_0}  e q^{-e }$ converges, we see that the limit as $e_0\ra\infty$ of the above is $0$.

Next, we bound the probability of a singularity at a point $Q$ with $\frac{\deg(Q)}{\deg(\pi(Q)}>\frac{n+1}{m+1}$.
We apply Lemma~\ref{L:2kindshigh} \eqref{E:highovermed}, with $j=e_0$ and $J=\lfloor \frac{n+1}{m+1} \rfloor +1$ to obtain that
for $n\geq 1$ and $d\geq 0$, we have that
\begin{align*}
&\frac{\# \{ f\in R_{n,d} | H_f \text{  sing. at $Q$ with $\deg(\pi(Q))\in[e_0,\lfloor d/p \rfloor +1] $
and $\frac{\deg(Q)}{\deg(\pi(Q))} > \frac{n+1}{m+1}$} \}}{\#R_{n,d}}\\
 &=O\left((n^m+d^m) q^{-d/p}+ 
 \sum_{e=e_0}^{\lfloor d/p \rfloor+1} e(n^m+e^m) q^{e\left(b-\min\left(\left\lfloor \tfrac{n+1}{m+1} \right\rfloor +1, \left\lfloor \tfrac{n-1}{p} \right\rfloor+1 \right)\right)}
\right) .
\end{align*}
We have that  $\lim_{d\ra\infty} (n^m+d^m) q^{-d/p}=0$.
For the second term, for $n\geq \max(b(m+1)-1,bp+1),$ we have
$$
b-\min(\lfloor \tfrac{n+1}{m+1} \rfloor +1, \lfloor \tfrac{n-1}{p} \rfloor+1  )<0,
$$
and thus the sum
$$
\sum_{e} e(n^m+e^m) q^{e(b-\min(\lfloor (n+1)/(m+1) \rfloor +1, \lfloor (n-1)/p \rfloor+1  ))}
$$
converges.
So for $n\geq \max(b(m+1)-1,bp+1),$
$$
\lim_{e_0\ra\infty} \lim_{d\ra\infty}\sum_{e=e_0}^{\lfloor d/p \rfloor+1} e(n^m+e^m) q^{e(b-\min(\lfloor (n+1)/(m+1) \rfloor +1, \lfloor (n-1)/p \rfloor+1  ))}
=0.
$$
We may thus set
\begin{equation}\label{eqn:choosing n0}
n_{0}=\max(b(m+1)-1,bp+1),
\end{equation}
completing the proof.
\end{proof}

\subsection{High degree points}
We apply Lemma~\ref{L:2kindshigh} to handle points of high degree in $B$.
\begin{proof}[Proof of Lemma~\ref{L:high}]
Fix some $n\geq 1$ and some $d\geq 0$ such that all points of $Z$ have degree at at most $\frac{d}{\max(M+1,p)}$.  Then  Lemma~\ref{L:2kindshigh} \eqref{E:overhigh}, with $j\geq \frac{d}{\max(M+1,p)}$ gives
\begin{align*}
 &\lim_{d\ra\infty} \frac{\#\{ f\in R_{n,d} | H_f \text{ is singular at a point $Q$
with $\deg(\pi(Q))> \frac{d}{\max(M+1,p)}$
}}{\#R_{n,d}}\\
=&\lim_{d\ra\infty} O((n^m+d^m) q^{-\min(\lfloor d/p \rfloor+1,d/\max(M+1,p))})=0.
\end{align*}
\end{proof}

\section{Convergence of the product}\label{S:convergence}

Since the product in Theorem~\ref{T:taylor} is a product of numbers between $0$ and $1$, it has a limit, and convergence of such a product refers to the limit being non-zero when all of the factors are non-zero.
In this section, we prove that convergence, and that for sufficiently large $n$, all factors are non-zero.

\begin{proposition}\label{P:converges}
For $n\geq n_0$, the product 
$$
\prod_{P\in B \setminus Z} \Prob\left(\begin{matrix} H_f \text{ is smooth at}\\ \text{ all points of } \pi^{-1}(P)\end{matrix}\right)
$$ in Theorem~\ref{T:taylor} is zero only if one of the factors is zero.
\end{proposition}

\begin{proof}
Given $P\in B_k\setminus Z$ with $\deg(P)=e$, we bound the probability $S_P$ that $H_f$ is singular at a point of $\pi^{-1}(P)$.  (Note that $S_P$ is a limit as $d\ra\infty$).
By Lemma~\ref{L:fiberbound}, the probability that $H_f$ is
singular at a point of $\pi^{-1}(P)$ of relative degree at most $(n+1)/(m+1)$ is $O(eq^{-e(k+1)})$.
 By Lemma~\ref{L:2kindshigh} \eqref{E:highoverP} with $J=\lfloor \frac{n+1}{m+1}\rfloor+1$, the probability that $H_f$ is
singular at a point of $\pi^{-1}(P)$ of relative degree at least $(n+1)/(m+1)+1$  is $O(e(n^m+e^m) q^{-e\min(\frac{n+1}{m+1}+1, \lfloor \frac{n-1}{p} \rfloor +1) )})$.
By the Lang--Weil bound \cite[Lemma 1]{LW}, $\# B_k(\F_{q^e}) = O(q^{ek})$.
Thus
$$
\sum_{P\in B_k \setminus Z} S_P = \sum_{e=1}^{\infty} O(q^{ek}eq^{-e(k+1)}+ q^{eb} e(n^m+e^m) q^{-e\min(\frac{n+1}{m+1}+1, \lfloor \frac{n-1}{p} \rfloor +1) )}),
$$
and for $n\geq n_0$ the sum on the right converges.  The proposition then follows by the fact that $\prod (1-a_i)$ converges if and only if $\sum a_i$ converges.
\end{proof}

\begin{proposition}\label{P:notzero}
For $n$ sufficiently large given $X,A,E$, the product 
$$
\prod_{P\in B \setminus Z} \Prob\left(\begin{matrix} H_f \text{ is smooth at}\\ \text{ all points of } \pi^{-1}(P)\end{matrix}\right)
$$ in Theorem~\ref{T:taylor} is non-zero.
\end{proposition}
\begin{proof}
 By Proposition~\ref{P:converges}, we are reduced to showing that all factors are non-zero.
Given a closed point $P\in B \setminus Z$ with $\deg(P)=e$, we bound the probabilty $S_P$ that $H_f$ is singular at a point of $\pi^{-1}(P)$ (which is a limit as $d\ra\infty$).
We divide into low, medium, and high degree points in $\pi^{-1}(P)$, where for this proposition we redefine
low, medium, and high degree (differently from the rest of the paper).  For some $r$, we define low degree points
to be points of degree at most $r$, medium degree to have degree in $(r,e\frac{n+1}{m+1}]$, and high degree to
be degree in $(e\frac{n+1}{m+1},\infty)$.
Let $\zeta_{X_P}(m+1)^{-1}=\delta$, and note $0<\delta< 1$.  

By Lemma~\ref{L:2kindshigh}\eqref{E:highoverP} with $J=\frac{n+1}{m+1}+1$, the probability that $H_f$ is
singular at a high degree point $X_P$ is $O(e(n^m+e^m) q^{-e\min(\frac{n+1}{m+1}+1, \lfloor \frac{n-1}{p} \rfloor +1) )})$.  For $n$ sufficiently large, this is $<\delta/2$ for all $e\geq 1$.

Let $Q$ be a medium degree point of degree $ef$.
The second statement of Lemma~\ref{L:sizeofimage}, applied with $W=Q^{(2)}$ (in $X$), gives
 that the probability of singularity at $Q$ is $q^{-ef(m+1)}$,
 and adding these up over medium degree points gives
us a sum bounded above by 
$$O\left(\sum_{f=\lfloor \frac{r}{e} \rfloor+1}^{\lfloor\frac{n+1}{m+1}\rfloor} q^{efm} q^{-ef(m+1)}\right).$$
We choose $r$ sufficiently large so that this bound on a medium degree singularity is $< \delta/2$ for all $e\geq 1$.

We have that $\prod_{Q\in X_P, \deg(Q)\leq r} (1-q^{-\deg(Q)(m+1)}) >\delta$.
Let $W$ be the union of all points of $X_P$ of degree at most $r$.
There are at most
$
 \deg_A(X_P) 2^{m+1} q^{mr}=O( q^{mr})
$
such points  by \cite[Lemma 1]{LW}, and so $W^{(2)}$ (in $X$) has degree that is $O(q^{mr})$.  By \cite[Lemma~2.1]{poonen}, for $n$ sufficiently large given the $r$ chosen above, we have that
$$
H^0(\P^N \times \P^M ,\cO(n,0))\ra H^0(W^{(2)},\cO_W^{(2)})
$$
is surjective
and thus 
$$
H^0(X ,\cO(nA))\ra H^0(W^{(2)},\cO_W^{(2)})
$$
is surjective and so is
$$
H^0(X ,\cO(nA+dE))\ra H^0(W^{(2)},\cO_W^{(2)}).
$$
Thus, the probabilities of singularities at low degree points are independent 
and the total probability of low degree singularity is  $1-\prod_{Q\in X_P, \deg(Q)\leq r} (1-q^{-\deg(Q)(m+1)})<1-\delta$.

Thus, for $n$ sufficiently large, we add up the low, medium, and high probabilities of singularity, and we conclude that the total probability of singularity at a point of $X_P$ is strictly less than $(1-\delta)+\frac{\delta}{2}+\frac{\delta}{2}=1$.
\end{proof}
The value  of $n_0$ in Theorem~\ref{T:taylor}
may not be sufficiently large for the conclusion of Proposition~\ref{P:notzero} to hold.
See~\S\ref{subsec:anti bertini}.

\section{Better bounds for $n_0$}\label{S:better bounds}

In specific cases, one can often get a  better bound for $n_0$ than that given in the statement of Theorem~\ref{T:taylor}.
In this section we give some such examples, where we replace Lemma~\ref{L:med} and Proposition~\ref{P:converges}, which are the source of $n_0$,
in the proof of Theorem~\ref{T:taylor}.

\begin{proposition}\label{P:n0is1}
If $\pi$ in Theorem~\ref{T:taylor} is an isomorphism on $X\setminus X_Z$, then 
taking $n_0=1$ suffices for the conclusion of Theorem~\ref{T:taylor}.
\end{proposition}

\begin{proof}
Since $\pi$ is an isomorphism on $X\setminus X_Z$ and $Z\ne \pi(X)$, we have that 
 $m=b\leq M$.
For a closed point $Q\in X\setminus X_Z$ with
$\deg(\pi(Q))(M+1)\leq d$ and $n\geq 0$,
we apply Lemma~\ref{L:sizeofimage} with 
$W=Q^{(2)}$ (so $r=w=\deg(Q)(m+1)$)
to conclude that 
the map
$$
H^0(X,\cO(nA+dE))\ra H^0(Q^{(2)},\cO_{Q^{(2)}})
$$
is surjective.

The singularity probability at a point $Q\in X\setminus X_Z$ with
$\deg(\pi(Q))(M+1)\leq d$ is then $q^{-\deg(Q)(m+1)}$.  Adding up over all such points
with $\deg(Q)\geq e_0$,
(using the Lang--Weil bound \cite[Lemma 1]{LW} to obtain $\#(X\setminus X_Z)(\F_{q^e})\in \cO(q^{me})$), we get a bound of
$$O\left(\sum_{e=e_0}^{\infty} q^{me}  q^{-e(m+1)}\right)= O(q^{-e_0})$$
on medium degree singularities for any $n\geq 0$.  Thus, in this case, Lemma~\ref{L:med} holds for $n\geq 0$.

Given $Q\in X\setminus X_Z$, the probabilty $S_Q$ that $H_f$ is singular at $Q$ (which is a limit as $d\ra\infty$)
is then $q^{-\deg(Q)(m+1)}$ for $n\geq 0$.  From this, Proposition~\ref{P:converges} follows easily for any $n\geq 0$.
\end{proof}

For example, if $D$ is very ample, then we can apply Proposition~\ref{P:n0is1} and see that $n_0$ is 1.

\subsection{Hirzebruch surfaces}\label{S:hirzebruch}

As a source of more examples, we consider Hirzebruch surfaces.  First we fix notation.
Let $X$ be the Hirzebruch surface $\P(\cO x\oplus\cO(a)y)$, a $\P^1$ bundle over $\P^1$, for $a\geq 0$. 
Let $D_v$ be a divisor pulled back from degree $1$ point on the base $\P^1$, and $D_h$
be a divisor in the class from the natural $\cO(1)$ that comes from the projective bundle construction (e.g. as in 
\cite[II.7]{Hart77}).  The Picard group of $X$ is $\Z^2$, and we will call divisors in the class
$iD_h+jD_v$ \emph{bi-degree} $(i,j)$.
The Cox Ring of $X$ is $\F_q[x,y,s,t]$, where $x$ has degree $(1,0)$ and $s,t$ have degree $(0,1)$, and $y$ has degree $(1,-a)$. 
The global sections of $\cO(iD_h+jD_v)$ are given by the polynomials of degree $(i,j)$.

If in Theorem~\ref{T:taylor}, we take $E=D_h,$ then for $a>0$ the map $\pi$ contracts a single $\P^1$ outside of which it
is an isomorphism.
  (For example, if $a=1$, then $X$ is the blow-up of $\P^2$ at a point, and $\pi$ is the map to $\P^2$.)
Then we can take $Z$ to be the image of the contracted $\P^1$ and apply Proposition~\ref{P:n0is1} to obtain that $n_0$ is 1.

If we take $E=D_v,$ then $\pi: X \ra \P^1$, and the following argument shows that $n_0$ is 1.
\begin{proposition}\label{P:hirzebruch}
For $X$ a Hirzebruch surface as defined above and $E=D_v,$ 
taking $n_0=1$ suffices for the conclusion of Theorem~\ref{T:taylor}.
\end{proposition}

\begin{proof}
Since $A$ is very ample, we may write $A=iD_h+jD_v$ for $i,j>0$ and thus $nA+dE=niD_h+(nj+d)D_v$.
So we consider sections $f$ of bidegree $(ni,nj+d)$.
For the purposes of this proof, we (possibly) enlarge the set of medium points by considering points where $\deg \pi(Q)\in [e_0, \frac{d}{2}]$.
We replace Lemma~\ref{L:med} and Proposition~\ref{P:converges} by bounding singularities in medium degree fibers in two cases.

First, we bound the probability that
for some closed point $P\in B$ with $\deg P\in [e_0,\frac{d}{2}]$, the hypersurface $H_f$ includes $X_P$ as a component.  The map  
\[
H^0(X,\cO_X(niD_h+(nj+d)D_v)) \ra H^0(X_P,\cO_{X_P}(niD_h+(nj+d)D_v))=H^0(\P^1_P,\cO_{\P^1_P}(ni)).
\]
is surjective for any $i$ and for $(nj+d)\geq \deg(P)-1$; this follows since the ideal sheaf of $X_P$ is $\cO(-\deg(P)\cdot D_v)$ and
\[
H^1(X,\cO_X(iD_h+(nj+d-\deg(P))D_v)=0 \text{ whenever } nj+d\geq \deg(P)-1.
\]
Thus, given $n\geq 1$ and  $P$ with $\deg(P)\leq \frac{d}{2}< nj+d$, the probability that $H_f$ includes $X_P$ as a component
is $q^{-\deg(P)(ni+1)}$.  Summing over $P$ of degree in the medium range $\left[e_0,\lfloor \frac{d}{2}\rfloor\right]$, we obtain a bound of 
$O(q^{-e_0})$, which goes to $0$ as $e_0 \ra \infty.$ 

Second, we bound the probability that for some closed point $P\in B$ with $\deg P\in [e_0,\frac{d}{2}]$,
$H_f$ is singular at some closed point $Q$ with $\pi(Q)=P$, and $H_f$ not containing $X_P$ as a component.
We assume $n\geq 0$.
By essentially the same analysis as in the proof of Lemma~\ref{L:sizeofimage}, the map
\[
H^0(X,\cO(niD_h+(nj+d)D_v)) \ra H^0(Q^{(2)},\cO_{Q^{(2)}})
\]
is surjective when $ni\geq 2 \frac{\deg(Q)}{\deg(P)}-1$ and
$nj+d\geq 2\deg(P)-1$.  (This can also be checked using the explicit
description of sections as above.)
Further, we note that $f$ has bidegree $(ni,nj+d)$ and if $H_f$ is singular at a point $Q$ with $\pi(Q)=P$ where $X_P$ is not a component of $H_f$, it follows by considering intersection multiplicities that
\[
2\deg(Q)\leq (niD_h+(nj+d)D_v)\cdot X_P=ni\cdot\deg(P);
\]
the coefficient of $2$ is because $H_f$ is singular at $Q$.
In particular, $\frac{\deg(Q)}{\deg(P)}\leq \frac{ni}{2}$ and thus, given $P$ with $\deg(P)\leq \frac{d}{2}$, and given $Q$ with $\pi(Q)=P$, the probability that $H_f$ is singular at $Q$ is at most $q^{-3deg(Q)}.$
Using the Lang--Weil bound \cite[Lemma 1]{LW} to add up over all points $Q\in X$
with $\deg(Q)\geq e_0$ (this will give an overestimate, since we need only sum over points which further
satisfy $\deg \pi(Q)\leq \frac{d}{2}$ and $\frac{\deg Q}{\deg \pi(Q)}\leq \frac{ni}{2}$), we obtain a bound of
$$O\left(\sum_{e=e_0}^{\infty}  q^{2e}  q^{-3e}\right).$$
This goes to $0$ as $e_0 \ra \infty,$ and thus Lemma~\ref{L:med} holds in this case for any $n\geq 1$.

Given $n\geq 1$ and $P\in B\setminus Z$ with $\deg(P)=e$, the probabilty $S_P$ that $H_f$ is singular at a point in $\pi^{-1}(P)$ (which is a limit as $d\ra\infty$) 
is then at most $q^{-2e}+\sum_{Q\in \pi^{-1}(P)} q^{-3\deg(Q)}=q^{-2e}+\sum_{k\geq 1} q^{-2ek}.$
 From this, Proposition~\ref{P:converges} follows easily for any $n\geq 1$.
\end{proof}

We conclude that, for any Hirzebruch surface $X$, with $E$ globally generated, we can always take $n_0=1$ in Theorem~\ref{T:taylor}, since $E$ is either a multiple of $D_h$ or a multiple of $D_v$, or else very ample.

\section{Further applications}\label{S:applications}
\subsection{Singularities of positive dimension}\label{subsec:pos dim}
Unlike~\cite[Theorem~3.2]{poonen}, the probability of having a singularity of positive dimension will generally be non-zero for the families of divisors
we consider.  For instance, on $\P^1\times \P^1$, there is a $q^{-2(n+1)}$ probability that a section of $\cO_{\P^1\times \P^1}(n,d)$ (as $d\ra\infty$) will contain a non-reduced copy of any given vertical fiber over an $\F_q$ point.

\subsection{Anti-Bertini Examples and a negative answer to a question of Baker}\label{subsec:anti bertini}

The value of $n_0$ coming from the proof of Theorem~\ref{T:taylor}
may not be sufficiently large for the resulting product to be non-zero.

\begin{proposition}[Fiber probabilities equal to $0$]\label{prop:fiber zero}
Fix $m>1$.  There exists a smooth variety $X$ of dimension $m$, with very ample $A$, globally generated $E$, and $P\in \pi(X)$, such that
the local probability of smoothness at points of $\pi^{-1}(P)$ equals $0$ for all $n\leq n_0$ (with $n_0$ as defined in Theorem~\ref{T:taylor}.)
\end{proposition}

We will deduce this from the following bigraded Anti-Bertini theorem, which proves the assertions in Example~\ref{E:multi anti}.
\begin{proposition}[Bigraded Anti-Bertini]\label{P:anti bertini}
For any $B,C\geq 1$, there exists a smooth variety $X\subseteq \P^B\times \P^C$ such that every hypersurface of bidegree $(n,n+d)$ has singular intersection with $X$ for all $n\leq n_0$ (with $n_0$ as defined in Theorem~\ref{T:taylor}) and all $d\geq n$.
\end{proposition}
\begin{proof}
We produce explicit examples in the style of Poonen's Anti-Bertini Theorem~\cite[Theorem 3.1]{poonen}.
Fix $B,C\geq 1$ and enumerate the hypersurfaces $H_i$ in $\P^B \times \P^C$ of bidegree $(j,j)$ for $j=1,\dots, \max(BC^2-1, Cp+1)$.
We fix a degree $1$ point $P$ in $\P^C$, and for each $H_i$, we choose a distinct closed point $Q_i$ in the fiber over $P$.
Then we can apply \cite[Theorem 1.2]{poonen} to find a smooth hypersurface $X \in \P^B \times \P^C$
of bidegree $(k,k)$, for some $k$, that contains all the $Q_i$ and at each $Q_i$ is tangent to $H_i$. 
In particular, $X\cap H_i$ is singular for all $i$ at the point $Q_i\in X_P$.  
 
We now take $A=\cO_X(1,1)$ and $E=\cO_X(0,1)$.  
By construction, for any $1\leq n\leq \max(BC^2-1, Cp+1)$, all of the sections of $H^0(\P^B \times \P^C,\cO(n,n))$ 
have singular intersection with $X$ at some point in $X_P$.  
We have a commutative diagram:
 \[
 \xymatrix{
 H^0(\P^B\times \P^C, \cO(n,n))\ar[r]^-\phi&H^0(\P^B\times P^{(2)}, \cO(n))\ar[r]^-\psi&H^0(X_{P^{(2)}}, \cO(nA)).\\
 H^0(\P^B\times \P^C, \cO(n,n+d))\ar[ru]^-\tau&&
 }
 \]
Since the ideal sheaf of $P^{(2)}$ in $\P^D$ has regularity $2$ (it is resolved by an Eagon--Northcott complex~\cite[\S A2H]{eisenbud-syzygies}), 
it follows that $\phi$ and $\tau$ are surjective if $n\geq 1$ and $n+d\geq 1$, respectively.  
Thus the image of $\psi\circ \tau$ equals the image of $\phi\circ \tau$, and for every $d\geq 0$, every section of  $H^0(\P^B\times \P^C, \cO(n,n+d))$
has singular intersection with $X$ at some point in $X_P$ for all $1\leq n\leq \max(BC^2-1, Cp+1)$.  
Since $\max(BC^2-1, Cp+1)$ is at least as large as the value of $n_0$ given in Theorem~\ref{T:taylor} for $X$, this proves the proposition.
\end{proof}

\begin{proof}[Proof of Proposition~\ref{prop:fiber zero}]
We fix $B>1$ and $C\geq 1$ and then choose $X,A,E$ as in the proof of Proposition~\ref{P:anti bertini}.
Since the ideal sheaf of $X$ is $\cO(-k,-k)$, there is a long exact sequence of the form
\[
\dots \to H^0(\P^B\times \P^C, \cO(n,n+d))\to H^0(X,\cO(nA+dE)) \to H^1(\P^B\times \P^C, \cO(n-k,n+d-k))\to \dots
\]
The first map is surjective whenever $d\geq k-n$.  Thus, when $B>1$ and $1\leq n\leq \max(BC^2-1, Cp+1)$, we see that every section of $H^0(X,\cO(nA+dE))$ has singular intersection with $X$ at a point in $X_P$ for all $d\geq k-n$.  
\end{proof}

Of course the above construction works with $\max(BC^2-1, Cp+1)$ replaced by any constant.
This construction leads directly to the following negative answer to Baker's question~\cite[Question~4.1]{poonen}.
\begin{proposition}[Asymptotic Anti-Bertini]\label{P:baker poonen}
Fix $m>0$.  There exists a smooth projective variety $X$ of dimension $m$ and a sequence of nondegenerate embeddings $\kappa_d\colon X\subseteq \P^{N_d}$ such that no hyperplane of $\P^{N_d}$ has a smooth intersection with $X$ (for any $d$) and the $N_d\to \infty$.
If $m>1$, then we may further impose the condition that each $\kappa_d$ is given by a complete linear series.
\end{proposition}
\begin{proof}
Let $B=m$ and $C=1$.  We construct $X,A,$ and $E$ as in the proof of Proposition~\ref{P:anti bertini}, and we fix $n=1$.
So $\dim(X)=m$.   
The image of the map
\[
\rho\colon H^0(\P^B\times \P^C, \cO(n,n+d))\to H^0(X,\cO(nA+dE))
\]
is a linear series that defines an embedding $\kappa_d$ of $X$ for any $d$.  By Proposition~\ref{P:anti bertini}, every hyperplane section of $\kappa_d\colon X\subseteq \P^{N_d}$ has singular intersection with $X$. 

The ideal sheaf of $X$ is $\cO(-k,-k)$, and thus a direct Hilbert function computation illustrates that the dimension of this linear series is
\[
N_d:=\binom{n+B}{B}\cdot \binom{n+d+C}{C} - \binom{n-k+B}{B}\cdot \binom{n+d-k+C}{C}
\]
which goes to $\infty$ as $d\to \infty$ since $k\geq 1$.

If $m>1$ and so $B>1$, then the proof of Proposition~\ref{prop:fiber zero} also illustrates that the map $\rho$ is surjective for $d\geq k-n=k-1$.  In these cases, $\kappa_d$ is given by a complete linear series.  Restricting attention to the $\kappa_d$ with $d\geq k-1$ thus proves the second statement.
\end{proof}
\begin{remark}
An alternate approach to finding counterexamples to~\cite[Question~4.1]{poonen} was suggested to us by Poonen,
and would work for any smooth quasi-projective $X$, but would not use complete linear series on $X$.
One could define maps from $X$ to projective
space by taking random sections of $|dA|$ (letting $d\to \infty$),
with conditions imposed to rule out smooth intersection with each
hyperplane, and then show that the map is an embedding with positive
probability.
\end{remark}

\subsection{Smooth curves in Hirzebruch surfaces and their rational points}\label{subsec:point counting}
We continue with the notation of \S \ref{S:hirzebruch} for a Hirzebruch surface $X$.
We first compute the probability that a curve in $X$ is smooth.  Note that, if the bidegree increases in an ample direction,
then we can apply Example~\ref{E:very ample} to conclude that
the probability of smoothness is $\zeta_X(3)^{-1}$.  So we focus on families of curves where the
bidegree grows in a semiample direction.  We then will describe the distribution of $\F_q$ points on smooth curves in $X$.

\begin{theorem}\label{T:Hsmooth}
For fixed $n\geq 3$ and $d\ra \infty$, the probability that a curve of bidegree $(n,d)$ in a Hirzebruch surface $X$ is smooth 
is
$$
\prod_{P \in \P^1_{\F_q}} (1-q^{-2\deg(P)})(1-q^{-3\deg(P)})=\zeta_{\P^1_{\F_q}}(2)^{-1}\zeta_{\P^1_{\F_q}}(3)^{-1}=
(1-q^{-1})(1-q^{-2})^2(1-q^{-3}).
$$
\end{theorem}

The following is an example application of our Theorem~\ref{T:Hpoints}, which appears below and gives the full distribution of $\F_q$ points of smooth curves.
\begin{corollary}\label{cor:avg 3d}
As $d\to\infty$, the following table gives the average number of points on a smooth curve of various types 
on a Hirzebruch surface:
\begin{center}
\renewcommand{\arraystretch}{1.2}
\begin{tabular}{ | c | c | } \hline
Curve bidegree & Average number of points\\ \hline
 $(2,d)$ & $q+2+\frac{1}{q^3+q^2-1}$ \\  \hline
 $(3,d)$ & $q+2-\frac{1}{q^2+q+1}$ \\  \hline
 $(d,2)$ & $q+2-\frac{q^3-q-1}{(q^3+q^2-1)(q^2+q+1)}$ \\  \hline
\end{tabular}
\end{center}
\end{corollary}
Note that these averages are close to $q+2$, as opposed to the same average for plane curves
or cyclic $p-$fold curves which is $q+1$ \cite{bdfl-planar,bdfl-trigonal, bdfl-pfold,  kurlberg-rudnick}
or for complete intersections in $\P^n$ which is near $q+1$ but smaller \cite{bucur-kedlaya}.
However, for trigonal curves the averages are near $q+2$ (see also the comment after the proof of Theorem~\ref{T:Hsmooth}).

To prove Theorem~\ref{T:Hsmooth}, we compute the product of Theorem~\ref{T:bertini}
when $A=nD_h+kD_v$ for some $n,k\geq 1$ and $E=D_v$. 
Let $P\in \pi(X)$ be a closed point of degree $e$.  Then the map
$$
H^0(X, \cO(nA+dE)) \ra H^0(X_{P^{(2)}}, \cO(nA))
$$
is surjective for $n\geq 1$ and $d$ sufficiently large.
Let $r(s)$ be an irreducible polynomial of $\F_q[s]$ of degree $e$.  Then $P^{(2)}\isom \Spec \F_q[s]/r(s)^2,$
and $X_{P^{(2)}} \isom \P^1_{\F_q[s]/r(s)^2}$.
Then, in Theorem~\ref{T:bertini}, the probability that $D$ is smooth at all points of $\pi^{-1}(P)$ is equal to the proportion
of sections of $H^0(\P^1_{\F_q[s]/r(s)^2}, \cO(n))$ that, for every closed point $Q\in \P^1_{\F_q[s]/r(s)^2}$, do not
vanish on $Q^{(2)}$.

Let $\tilde{~}: \F_q[s]/r(s) \ra \F_q[s]/r(s)^2$ be a $\F_q$-vector space map that is a section of the reduction map $\overline{~} : \F_q[s]/r(s)^2\ra\F_q[s]/r(s).$
We call a pair $(F_1,F_2)\in   H^0(\P^1_{\F_{q^e}}, \cO(i)) \times H^0(\P^1_{\F_{q^e}}, \cO(j))$ \emph{good} if 
$F_1$ is non-zero and
there is no geometric point of $\P^1_{\F_{q^e}}$ that is a multiple root of $F_1$ and a root of $F_2$.

\begin{lemma}\label{L:fbij}
The map
$$
H^0(\P^1_{\F_{q^e}}, \cO(n))^2 \ra H^0(\P^1_{\F_q[s]/r(s)^2}, \cO(n))
$$
given by $(F_1,F_2) \mapsto \tilde{F_1} + r(s) F_2$ (applying $\tilde{~}$ coefficient-wise)
is a bijection.  Moreover for $n\geq 1$,  it restricts to a bijection
from good
pairs to sections
$f\in H^0(\P^1_{\F_q[s]/r(s)^2}, \cO(n))$ that, for every closed point $Q\in \P^1_{\F_q[s]/r(s)^2}$, do not
vanish on $Q^{(2)}$.
\end{lemma}

\begin{proof}
The map given in the lemma has an inverse $f\mapsto (\overline{f},  \frac{ f-\widetilde{\overline{f}}}{r(s)})$.
To check whether an $f\in H^0(\P^1_{\F_q[s]/r(s)^2}, \cO(n))$ vanishes on $Q^{(2)}$, we check whether $f$ and its two partial derivatives all vanish at $Q$.  We can replace the local parameter $s$ by $r(s)$, and thus we are checking the simultaneous vanishing of $F_1, F_1',$ and  $F_2$ at the image of $Q$ in $\P^1_{\F_q[s]/r(s)}$.
For $n\geq 1$, since $F_2$ has some root, goodness requires that $F_1$ is non-zero.
\end{proof}

\begin{lemma}\label{L:countgood}
The number of good pairs $(F_1,F_2)\in H^0(\P^1_{\F_{q^e}}, \cO(n))^2$
is
$$ q^{e(2n+2)} \cdot
\begin{cases}
1-q^{-2e} &\textrm{ if } n=1\\
1-q^{-2e}-q^{-3e}+q^{-4e}  &\textrm{ if } n=2\\
1-q^{-2e}-q^{-3e}+q^{-5e} &\textrm{ if } n\geq 3.
\end{cases}
$$
\end{lemma}

\begin{proof}
Let $G_{i,j}$ be the number of good pairs in $H^0(\P^1_{\F_{q^e}}, \cO(i)) \times H^0(\P^1_{\F_{q^e}}, \cO(j))$.
Given any pair  $(F_1,F_2) \in \left(H^0(\P^1_{\F_{q^e}}, \cO(i))\setminus\{0\} \right)\times H^0(\P^1_{\F_{q^e}}, \cO(j))$,
there is a unique divisor $D \sub \P^1_{\F_{q^e}}$ such that if $s_D$ is a section determining $D$ then
$(F_1/s_D^2, F_2/s_D)$ is good.  This construction gives the equality of generating functions
$$
 \sum_{i\geq 0} {(q^{e(i+1)}-1)}t^i   \sum_{j\geq 0} q^{e(j+1)}s^j  =
 \sum_{k\geq 0} \frac{q^{e(k+1)}-1}{q^e-1}t^{2k}s^k  \sum_{i,j\geq 0} G_{i,j} t^is^j 
$$
and thus
$$
\sum_{i,j\geq 0} G_{i,j} t^is^j =
(1-q^et^2s)(1-t^2s)\sum_{i\geq 0} {(q^{e(i+1)}-1)}t^i   \sum_{j\geq 0} q^{e(j+1)}s^j,
$$
from which the lemma follows.
\end{proof}

Thus, e.g. if $n\geq 3$,  the probability that $D$ is smooth at all points of $\pi^{-1}(P)$ is $1-q^{-2e}-q^{-3e}+q^{-5e}$, and so
we conclude Theorem~\ref{T:Hsmooth}.
The fact that for $n\geq 3$ this probability does not depend on $n$ is rather special to $\P^1$.  
A version of Poonen's Bertini theorem \cite[Theorem 1.1]{poonen} over a local ring with residue field $\F_q$
would imply that these probabilities have a limit as $n \ra \infty$, but in general one does not expect this limit to
be reached for all sufficiently large $n$.  However, the case of $\P^1$ is very special, and 
for example the proportion of sections of $H^0(\P^1_{\F_q},\cO(n))$ that give smooth divisors is 
$1-q^{-1}-q^{-2}+q^{-3}$ for all $n\geq 3$ \cite[Proposition 5.9(a)]{vakil-wood}, i.e. the limit in \cite[Theorem 1.1]{poonen} is obtained for all degrees at least $3$.

Now we can give the distribution of $\F_q$ rational points on a random curve in $X$ in various bidegree families.

\begin{theorem}\label{T:Hpoints} \ 
\begin{enumerate}
 \item \label{ample}

For fixed, $i,j,k,\ell$ we have
$$\lim_{d\ra\infty} \frac{\#\{\textrm{smooth }C\sub X \textrm{ of bidegree }(id+j,kd+\ell) \mid \#C(\F_q)=k  \}}{\#\{\textrm{smooth }C\sub X \textrm{ of bidegree }(id+j,kd+\ell)\}}=
\operatorname{Prob}\left(\sum_{i=1}^{(q+1)^2} X_i=k\right),
$$
where the $X_i$ are independent identically distributed random variables and
$$
X_i=
\begin{cases}
0 &\textrm{ with probability } \frac{q^3-q^2}{q^3-1} \\
1 &\textrm{ with probability } \frac{q^2-1}{q^3-1}.
\end{cases}
$$
\item\label{2d}
We have
$$\lim_{d\ra\infty} \frac{\#\{\textrm{smooth }C\sub X \textrm{ of bidegree }(2,d) \mid \#C(\F_q)=k  \}}{\#\{\textrm{smooth }C\sub X \textrm{ of bidegree }(2,d)\}}=
\operatorname{Prob}\left(\sum_{i=1}^{q+1}Y_{i}=k\right),
$$
where the $Y_i$ are independent identically distributed random variables and
$$
Y_i=
\begin{cases}
0 &\textrm{ with probability } \frac{q^3-q^2}{2q^3+2q^2-2} \\
1 &\textrm{ with probability } \frac{2q^2-2}{2q^3+2q^2-2}\\
2 &\textrm{ with probability } \frac{q^3+q^2}{2q^3+2q^2-2}.
\end{cases}
$$
\item\label{3d}
We have
$$\lim_{d\ra\infty} \frac{\#\{\textrm{smooth }C\sub X \textrm{ of bidegree }(3,d) \mid \#C(\F_q)=k  \}}{\#\{\textrm{smooth }C\sub X \textrm{ of bidegree }(3,d)\}}=
\operatorname{Prob}\left(\sum_{i=1}^{q+1} Z_i=k\right),
$$
where the $Z_i$ are independent identically distributed random variables and
$$
Z_i=
\begin{cases}
0 &\textrm{ with probability } \frac{2q^2}{6q^2+6q+6} \\
1 &\textrm{ with probability } \frac{3q^2+6}{6q^2+6q+6}\\
2 &\textrm{ with probability } \frac{6q}{6q^2+6q+6}\\
3 &\textrm{ with probability } \frac{q^2}{6q^2+6q+6}.
\end{cases}
$$
\item\label{d2}
For $X\not\isom \P^1 \times \P^1$, we have
$$\lim_{d\ra\infty} \frac{\#\{\textrm{smooth }C\sub X \textrm{ of bidegree }(d,2) \mid \#C(\F_q)=k  \}}{\#\{\textrm{smooth }C\sub X \textrm{ of bidegree }(d,2)\}}=
\operatorname{Prob}{\left(Y_1+\sum_{i=1}^{q^2+q} X_i=k\right)}
$$
where the $X_i$ and $Y_1$ are independent  random variables defined above.
\item \label{d3}
For $X\not\isom \P^1 \times \P^1$, we have
$$\lim_{d\ra\infty} \frac{\#\{\textrm{smooth }C\sub X \textrm{ of bidegree }(d,3) \mid \#C(\F_q)=k  \}}{\#\{\textrm{smooth }C\sub X \textrm{ of bidegree }(d,3)\}}=
\operatorname{Prob}\left(Z_1+\sum_{i=1}^{q^2+q} X_i=k\right),
$$
where the $X_i$ and $Z_1$ are independent  random variables defined above.
\end{enumerate}
\end{theorem}

\begin{proof}
We apply Theorem~\ref{T:taylor}
(with $A$ and $E$ chosen to give the relevant bidegrees)
 with $Z$ the union of degree $1$ points in $\pi(X)$, and $T$ corresponding to divisors that are smooth at all points of $\pi^{-1}(Z)$ and pass through a fixed number of degree $1$ points of $\pi^{-1}(Z)$.
We then must divide the result by the result of Theorem~\ref{T:taylor} for $Z$ empty, and so the only factors in Theorem~\ref{T:taylor} to be computed are those at the degree $1$ points in $\pi(X)$.

For \eqref{ample}, this computation is simple, because for a closed point $Q\in X$ and $d$ sufficiently large, $H^0(X,\cO(nA+dE))\ra H^0(Q^{(2)},\cO_{Q^{(2)}})$ is surjective (as mentioned in Example~\ref{E:very ample}), and
we have the same analysis as in \cite[Theorem~1.1]{bdfl-planar}.

Aided by Lemmas~\ref{L:fbij} and \ref{L:countgood} (though this computation is simple even without the lemmas), we can
count the proportion of $H^0(\P^1_{\F_q[s]/s^2}, \cO(n))$ that don't vanish at any $Q^{(2)}$ and pass through
a fixed number of degree $1$ points, for $n=2,3$ to conclude \eqref{2d} and \eqref{3d}.

For \eqref{d2} and \eqref{d3}, the analysis at the $q^2+q$ points of $X$ on which $\pi$ is an isomorphism is the same 
as in the case \eqref{ample}, and the analysis on the $\P^1$ contracted by $\pi$ is very similar to \eqref{2d} and \eqref{3d}, respectively.
\end{proof}

This distribution of points for $(3,d)$ curves is identical the the distribution of points on trigonal curves over $\F_q$
as the genus goes to infinity \cite[Theorem 1.1]{wood-trigonal}, which is not surprising as each trigonal curve
can be embedded in a unique Hirzebruch surface as a bidegree $(3,d)$ curve.  However, \cite[Theorem 1.1]{wood-trigonal}
(proven with function field methods) does not follow immediately from Theorem~\ref{T:Hpoints} above because
it is not clear how adding over all Hirzebruch surfaces interacts with the limit in $d$.

\subsection{A quasi-projective version of Theorem~\ref{T:taylor}} 
We let $X, A$ and $E$ be defined in \S\ref{S:notation} and we let $X^\circ\subseteq X$ be an open subset.  We set $B^\circ:=\pi(X^\circ).$
When we shift to the quasi-projective case, there may be a non-zero probability that $f\in R_{n,d}$ is a unit on $X^\circ$, and hence it might define
the empty set.  We allow this possibility, taking the convention that the empty set is smooth.
\begin{theorem}[Quasi-projective version of Theorem~\ref{T:taylor}]
Let $X, X^\circ, A,$ and $E$ as above.  Let $Z\subsetneq \pi(X^\circ)$ be a finite subscheme.  Assume that $X^\circ \cap \left(X \setminus \pi^{-1}(Z)\right)$ is smooth and let
$n_{0}:=\max(b(m+1)-1,bp+1)$, with the constants as defined in \S\ref{S:notation}.

For all $n\geq n_0$ and for all $T\subseteq H^0(\pi^{-1}(Z),\cO_{\pi^{-1}(Z)}(nA))$, and for a random  section $f\in R_{n,d}$ as $d\ra\infty$, we have
\[
\Prob\left(\begin{matrix} H_f \cap X^\circ \cap (X \setminus \pi^{-1}(Z)) \\ \text{is smooth and } f|_{\pi^{-1}(Z)}\in T\end{matrix} \right)=
\Prob(f|_{\pi^{-1}(Z)}\in T)
 \prod_{P\in B^\circ \setminus Z} \Prob\left(\begin{matrix} H_f \cap X^\circ \text{ is} \\ \text{smooth at all points}\\ \text{of } \pi^{-1}(P)\cap X^\circ\end{matrix}\right).
\]
The product over $P\in B^\circ \setminus Z$ converges, is zero only if some factor is zero, and is always non-zero for $n$ sufficiently large.
\end{theorem}
\begin{proof}[Sketch of proof]
The proof of Theorem~\ref{T:taylor} for the equality of probabilities goes through with only standard modifications (such as restricting attention to points in $X^\circ$ for the various lemmas in \S\ref{S:setup} and \S\ref{S:lmh}).
Further, since the product on the right involves less terms than the product on the right in Theorem~\ref{T:taylor}, the convergence statements are immediate.
\end{proof}

\subsection{Computing the factors in the product of Theorem~\ref{T:taylor}}\label{subsec:sections from PNPM}
When computing 
the values of the factors in the product of Theorem~\ref{T:taylor},
we can introduce some simplifications by (possibly)
increasing $n_0$.
For example, we can work explicitly with sections from $\P^N\times \P^M$ instead of from $X$ (see Lemma~\ref{L:regbounds}
\eqref{B:lem:section equality}).
Also, we can arrange so that, as $d\to \infty$,
the map stabilizes to a surjective map
\[
H^0(X,\cO_X(nA+dE))\to H^0(X_{P^{(2)}}, \cO_{X_{P^{(2)}}}(nA)),
\]
so it remains to determine which sections of $H^0(X_{P^{(2)}}, \cO_{X_{P^{(2)}}}(nA))$ don't vanish in any first order
infinitesimal nieghborhood (see Lemma~\ref{lem:stabilization}).

\begin{lemma}\label{L:regbounds}
We have the following uniform bounds:
\begin{enumerate}
	\item\label{B:P2regnew}  There exists $n_2$ (a function of $X,A,E$) such that $\reg_A \II_{X_{\olP^{(2)}}\subseteq \P^N} \leq n_2$ for all geometric points $\olP\in B\tensor_{\F_q} \overline{\F}_q$.
	\item\label{B:lem:section equality} There exist $n_3$ and $d_3$ (functions of $X,A,E$) such that we have a surjection
\[
H^0(\P^N \times \P^M,\cO(n,d)) \ra H^0(X,\cO(nA+dE))
\]
for all $n\geq n_3$ and all $d\geq d_3$.
\end{enumerate}
\end{lemma}

\begin{proof}
Let $\II_{X_{\olP}}$ be the ideal sheaf of $X_{\olP}\subseteq \P^N$.  The proof of Lemma~\ref{L:Global1}
gives the existence of a global bound on the degrees of generators of the ideals $\II_{X_{\olP}}$.  
To obtain $n_2$, we first need a global bound on the degrees of the generator that any ideal sheaf $\II_{X_{\olP^{(2)}}}$, and this follows from the global bound
for the reduced fibers combined with \cite[Theorem~3.1]{swanson}. Hence, there is a global bound on the degrees of the generators of the ideals $\II_{X_{\olP^{(2)}}}$, and this may be used to obtain a global bound on their regularity by~\cite[Proposition~3.8]{bayer-mumford}.

For part \eqref{B:lem:section equality}, we let $\mathcal J_X$ be the ideal sheaf for $X\subseteq \P^N\times \P^M$.  Let $\bF$ be a resolution of $\mathcal J_X$ via direct sums of line bundles on $\P^N\times \P^M$.  Then choose $(n_3,d_3)$ so that $n_3$ is larger than any $x$-degree that appears in $\bF$ and $d_3$ is larger than any $y$-degree that appears in $\bF$.    It follows that $\mathcal J_X(n,d)$ has no higher cohomology when $n\geq n_3$ and $d\geq d_3$, implying the desired surjectivity.
\end{proof}

\begin{lemma}\label{lem:stabilization}
Let $V\subsetneq B$ be a reduced $0$-dimensional scheme and  let $W=V^{(2)}$.  Then
\[
H^0(\P^N\times\P^M,\cO(n,d))\ra H^0(X_W,\cO_{X_W}(nA))
\]
is surjective for all $n\geq \max\{n_2,n_3\}$ and $d\geq \max\{\deg(W)-1,d_3\}$ (with all constants
as defined in Lemma~\ref{L:regbounds}).
Thus 
$
H^0(X,\cO(nA+dE))\ra H^0(X_W,\cO_{X_W}(nA))
$
is also surjective. 

\end{lemma}

\begin{proof}
We consider the commutative square:
\[
\xymatrix{
H^0(\P^N\times \P^M, \cO(n,d)) \ar[r]^-\tau\ar[d]^-{\rho}& H^0(X,\cO(nA+dE))\ar[d]^-{\rho'}\\
H^0(\P^N\times W, \cO(n)) \ar[r]^-{\tau'}& H^0(X_W, \cO_{}(nA)).
}
\]
Since $n\geq n_3$ and $d\geq d_3$, $\tau$ is surjective by Lemma~\ref{L:regbounds}.  Whenever $d\geq \deg(W)-1$, it follows (by~\cite[Lemma~2.1]{poonen} for example) that $\rho$ is surjective.  Thus $\im(\rho')=\im(\tau')$ and it suffices to prove surjectivity of $\tau'$.

It further suffices to prove the statement after passing to the algebraic closure, so we henceforth work over $\Fbar_q$.  Since $V$ is reduced, after passing to the algebraic closure, we have $V=\{P_1, \dots, P_{\deg V}\}$, and we can factor the above map in a different way:
\[
H^0(\P^N\times \P^M, \cO(n,d))\longrightarrow \bigoplus_{j=1}^{\deg V} H^0(\P^N\times P_j^{(2)}, \cO_{\P^N\times P_j^{(2)}}(n))\longrightarrow \bigoplus_{j=1}^{\deg V} H^0(X_{P_j^{(2)}}, \cO_{X_{P_j^{(2)}}}(nA)).
\]
The first map is surjective since $d\geq \deg(W)-1$.
To see that the second map is surjective, it suffices to check this in each factor.  And then we have that
\[
H^0(\P^N,\cO(n)) \to H^0(X_{P_j^{(2)}}, \cO_{X_{P_j^{(2)}}}(nA))
\]
is surjective by Lemma~\ref{L:regbounds} \eqref{B:P2regnew} since $n\geq n_2$.
\end{proof}

\bibliographystyle{alpha} 
\bibliography{myrefs.bib}

\end{document}